\documentclass[a4paper,14pt, reqno]{amsart}
\usepackage{bbm}
\numberwithin{equation}{section}
\usepackage{mathdots}
\usepackage{leftidx}
\usepackage{mathrsfs}
\usepackage{amsmath,amssymb, amsthm}
\usepackage[all,poly,knot]{xy}
\usepackage[dvips]{graphicx}
\usepackage[colorlinks,linkcolor=black,anchorcolor=black,citecolor=black]{hyperref}
\usepackage{color}
\usepackage{tikz}
\usepackage{picture}
\usepackage{caption}
\usepackage[all]{xy}
\usepackage{enumerate}
\usepackage{makecell}
\usepackage{mathtools}
	\usepackage{ulem}
\usepackage{extarrows}
\usepackage{graphicx}
\usepackage{caption}
\usepackage{subcaption}
\usepackage{float}
\usepackage{blkarray}
\numberwithin{equation}{section}
\newtheorem{Thm}{Theorem}[section]
\newtheorem{Lem}[Thm]{Lemma}
\newtheorem{Def}[Thm]{Definition}
\newtheorem{Cor}[Thm]{Corollary}
\newtheorem{Prop}[Thm]{Proposition}
\newtheorem{Ex}[Thm]{Example}
\newtheorem{Rem}[Thm]{Remark}

\newcommand{\bk}{{\bf k}}
\newcommand{\bm}{{\bf m}}
\newcommand{\val}{{\bf val}}
\newcommand{\grd}{{\bf grd}}
\newcommand{\ang}[1]{\left\langle #1 \right\rangle}

\DeclareMathOperator{\rad}{rad}
\DeclareMathOperator{\soc}{soc}
\def\gr{\mathrm{gr}}
\def\Ga{\Gamma}

\usepackage{mathrsfs}
\usepackage{calc,pifont}

\newcommand\defeqi{\stackrel{\text{$i$}}{\rule[0.6mm]{0.6cm}{0.1mm}}}
\newcommand\defeqj{\stackrel{\text{$j$}}{\rule[0.6mm]{0.6cm}{0.1mm}}}

	\title[Rogue wave]{Some results on representation type and Galois covering with application to the associated graded algebras of biserial fractional Brauer graph algebras}
	
	\author[Jing Guo]{Jing Guo}
	\address{
		School of Mathematics and Statistics\\
		Qingdao University \\Qingdao\\
		Shandong 266071   \\
		P.R.China}
	\email{jingguo@qdu.edu.cn}
	\author[Nengqun Li]{Nengqun Li*}
	\thanks{*Corresponding author.}
	\address{School of Mathematics\\ Liaoning Normal University\\ Dalian 116029\\ P.R.China}
	\email{linengqun@lnnu.edu.cn}
	\author[Yuming Liu]{Yuming Liu}
	\address{School of Mathematical Sciences\\ Laboratory of Mathematics and Complex Systems\\ Beijing Normal University\\Beijing 100875  \\ P.R.China}
	\email{ymliu@bnu.edu.cn}
	\date{}				
\begin{document}
	
\begin{abstract}
We prove some general results on representation types of special biserial algebras and on Galois coverings of quivers with relations. Moreover, combining these results with the covering theory of biserial fractional Brauer graph algebras we determine the representation types of the associated graded algebras of biserial fractional Brauer graph algebras.
\end{abstract}
	
\maketitle \vspace{-0.9cm}
	
\section{Introduction}
	
Following Drozd, all finite-dimensional $\bk$-algebras over an algebraically closed field $\bk$ split into two classes: those having tame representation type and those having wild representation type. In order to understand the representation theory of a class of finite-dimensional algebras, the first step is often to describe their representation types. The representation types of several classes of finite-dimensional algebras are well-studied, among them is the class of special biserial algebras.

Special biserial algebras are defined by quivers with relations and have tame representation type (see for example, a brief introduction in \cite{E}). Since there are different levels on tameness: finite representation type, domestic type, type of polynomial growth, type of nonpolynomial growth. One would like to know when a given special biserial algebra have a specific level of tameness. For self-injective special biserial algebras, a general answer for this question was given by Erdmann and Skowro\'{n}ski in \cite{ES}. In particular, they showed that a self-injective special biserial algebra is representation-infinite domestic if and only if it is representation-infinite of polynomial growth if and only if the number of equivalence classes of bands in the associated string algebra is a positive integer.

The first main result in the present paper generalizes the above result to arbitrary special biserial algebras:

\begin{Thm}[see Theorem \ref{Theorem}]
Let $A=\bk Q/I$ be a special biserial algebra. The following are equivalent:
\begin{enumerate}[$(i)$]
\item $A$ is representation-infinite domestic;
\item $A$ is representation-infinite of polynomial growth;
\item The cardinality of $\mathrm{Ba}(\overline{A})$ is a positive integer, where $\overline{A}$ is the associated string algebra of $A$.
\end{enumerate}
Moreover, if the cardinality of $\mathrm{Ba}(\overline{A})$ is $n$, then $A$ is $n$-domestic.
\end{Thm}

This general result applies in many not necessarily self-injective special biserial algebras, including the associated graded algebras of symmetric special biserial algebras which are studied in \cite{GL,GLY}. Symmetric special biserial algebras are also called Brauer graph algebras (abbr. ${\bf BGA}$s) since they can be defined by combinatorial data - Brauer graphs (abbr. ${\bf BG}$s). Recently, a general class of self-injective algebras was introduced and studied in \cite{LL1,LL2,LL3}, namely the fractional Brauer configuration algebras of type $S$ (abbr. {\bf $f_{s}$-BCA}s). A small subclass of {\bf $f_{s}$-BCA} is fractional Brauer graph algebras of type $MS$ (abbr. {\bf $f_{ms}$-BGA}s), which is named in the present paper as biserial fractional Brauer graph algebras (abbr. ${\bf bFBGA}$s). Since ${\bf bFBGA}$s are special biserial and proper generalization of ${\bf BGA}$s, they have received a lot of attention and research (see for example, \cite{LL3,X}). Similar as a ${\bf BGA}$, each ${\bf bFBGA}$ is defined by some combinatorial data - biserial fractional Brauer graph (abbr. ${\bf bFBG}$).

Given a ${\bf bFBG}$ $\Gamma$, denote by $\nu$ its Nakayama automorphism, then we can define a quotient $\Gamma/\ang{\nu}$ (which is a modified Brauer graph or ${\bf mBG}$ for short), which induces a ${\bf BG}$ $R_{\Gamma}$ called the reduced form of $\Gamma$. Denote by $A_{\Gamma}$ (resp. $A_{R_{\Gamma}}$) the corresponding ${\bf bFBGA}$ (resp. ${\bf BGA}$). Using covering of quivers with relations, Li and Liu proved that $A_{\Gamma}$ and $A_{R_{\Gamma}}$ have the same representation type:

\begin{Thm} [see {\cite[Theorem 2.35]{LL3}}]
Let $\Gamma$ be a ${\bf bFBG}$. Then $A_{\Gamma}$ is of finite representation type (resp. domestic) if and only if $A_{R_{\Gamma}}$ is of finite representation type (resp. domestic).
\end{Thm}

Note that however, according to Theorem 1.1, the above result does not mean that $A_{\Gamma}$ is $n$-domestic if and only if $A_{R_{\Gamma}}$ is $n$-domestic. In fact, if a ${\bf BGA}$ is $n$-domestic, then $n$ is less than or equal to $2$, but an $n$-domestic ${\bf bFBGA}$ can have arbitrary large $n$. In the following, when we say two (tame) algebras have the same representation type, we always mean that they are both of finite representation type, or both of domestic type, or both of polynomial growth, or both of nonpolynomial growth.

Since we have described the representation types of the associated graded algebras of ${\bf BGA}$s in \cite{GL,GLY}, it is natural to ask whether we can describe the representation types of the associated graded algebras of ${\bf bFBGA}$s using covering theory. In order to answer this question, we need to prove two general results on Galois covering of quivers with relations:
a Galois covering between two finite-dimensional algebras induces Galois coverings between the associated graded algebras with the radical filtration and between the associated quotient algebras by the ideals of the sum of socles of some non-uniserial indecomposable projective-injective modules. Based on these two results we get our second main result:

\begin{Thm}[see Corollary \ref{cor-induced-covering}]
	Let $A=\bk Q_A/I_A$ and $B=\bk Q_B/I_B$ be two special biserial algebras, $\overline{A}$ and $\overline{B}$ be the associated string algebras of $A$ and $B$ respectively. Let $\gr(A)$ and $\gr(B)$ be the associated graded algebras of
	$A$ and $B$ respectively, and let   $\overline{\gr(A)}$ and $\overline{\gr(B)}$ be the associated string algebras. If $f: (Q_A, I_A)\longrightarrow(Q_{B},I_{B})$ is a Galois covering  of quivers with relations with group $G$, then there exist ideals $I_{A,1}$, $I'_A$, $I_{A,2}$ in $\bk Q_A$ and ideals $I_{B,1}$, $I'_B$, $I_{B,2}$ in $\bk Q_B$ such that  $\overline{A}=\bk Q_A/I_{A,1}$, $\gr(A)=\bk Q_A/I_A'$, $\overline{\gr(A)}=\bk Q_A/I_{A,2}$ and $\overline{B}=\bk Q_B/I_{B,1}$, $\gr(B)=\bk Q_B/I_B'$, $\overline{\gr(B)}=\bk Q_B/I_{B,2}$, such that $f$ induces Galois coverings $(Q_A, I_{A,1})\longrightarrow(Q_{B},I_{B,1})$, $(Q_A, I'_A)\longrightarrow(Q_{B},I'_{B})$ and
	$(Q_A, I_{A,2})\longrightarrow(Q_{B},I_{B,2})$ with the same group. Therefore, 	
	\begin{enumerate}[$(1)$]
		\item  $\mathrm{Ba}(\overline{B})$ (resp. $\mathrm{Ba}(\overline{\gr(B)})$) is empty if and only if $\mathrm{Ba}(\overline{A})$ (resp. $\mathrm{Ba}(\overline{\gr(A)})$) is empty,
		\item The cardinality of $\mathrm{Ba}(\overline{B})$ (resp. $\mathrm{Ba}(\overline{\gr(B)})$) is a positive number if and only if the cardinality of $\mathrm{Ba}(\overline{A})$ (resp. $\mathrm{Ba}(\overline{\gr(A)})$) is a positive number.
	\end{enumerate}
In particular, $A$ and $B$ (resp. $\gr(A)$ and $\gr(B)$) have the same representation type.
\end{Thm}

We now apply the above results together with the covering theory of ${\bf bFBGA}$s developed in \cite{LL2,LL3} to describe the representation types of the associated graded algebras of ${\bf bFBGA}$s:

\begin{Thm}[see Theorem \ref{Thm-rep-typ}] \label{main-result}
	Let $A$ be a ${\bf bFBGA}$ associated with some ${\bf bFBG}$ $\Ga$ and $\mathrm{gr}(A)$ the graded algebra associated with the radical filtration of $A$. Suppose that $\nu$ is the Nakayama automorphism of $\Ga$ and $R_{\Gamma}$ is the reduced form of $\Ga$. Let $B$ be the ${\bf BGA}$ associated with $R_{\Gamma}$ and $\gr(B)$ the graded algebra associated with the radical filtration of $B$.  		
	Then we have the following.
\begin{enumerate}[$(1)$]
\item If $\Ga/\ang{\nu}$ is one of the four cases: (i) a ${\bf BT}$; (ii) a representation-finite modified Brauer tree (abbr. ${\bf mBT}$); (iii) a domestic ${\bf BG}$; (iv) a domestic ${\bf mBT}$, then
	\begin{enumerate}
		\item [$(1.1)$] $\gr(A)$ is of finite representation type if and only if $\gr(B)$ is of finite representation type.
		\item [$(1.2)$] the following are equivalent:
		\begin{enumerate}
			\item [$(1.2.1)$] 	$\gr(A)$ is domestic.
			\item [$(1.2.2)$]  $\gr(A)$ is of polynomial growth.
			\item [$(1.2.3)$] $\gr(B)$ is  domestic.
			\item [$(1.2.4)$]  $\gr(B)$ is of polynomial growth.
		\end{enumerate}
\end{enumerate}
\item If $\Ga/\ang{\nu}$ does not belong to the four cases in $(1)$, then both $\gr(A)$ and $\gr(B)$ are nondomestic and not of polynomial growth.
\end{enumerate}
\end{Thm}

 In above presentation, the quotient $\Gamma/\langle\nu\rangle$ (see the description after Definition \ref{def-bg-with-orbifolds}) plays a crucial role. By the covering theory in \cite{LL2,LL3}, for a ${\bf bFBG}$ $\Gamma$, $\Ga/\ang{\nu}$ is an ${\bf mBG}$. When the corresponding ${\bf bFBGA}$ $A_{\Gamma}$ is representation-finite or representation-infinite domestic, $\Ga/\ang{\nu}$ is either a ${\bf BG}$ or an ${\bf mBT}$ (cf. Theorems \ref{Thm-finite-f-BG} and \ref{Thm-domestic-f-BG} below). Indeed, we can describe the representation-finite and representation-infinite domestic $\gr(A)$ directly using the quotient $\Ga/\ang{\nu}$ (see Theorem \ref{thm}).

In appendix, we explicitly describe the defining ${\bf bFBG}$ $\Gamma$ of representation-finite or representation-infinite domestic graded ${\bf bFBGA}$ $\gr(A_{\Gamma})$. These results are given along Propositions \ref{rep-finite-(a)}-\ref{domestic-(e)} without proofs. The main idea is to construct the defining ${\bf bFBG}$ $\Ga$ using the quotient $\Ga/\ang{\nu}$ and the order of the Nakayama automorphism $\nu$ of $\Ga$.

This paper is organized as follows.

In Section \ref{rep.type-sba}, we first recall the definitions of special biserial algebras and representation types of finite-dimensional algebras, and then we state and prove our first main result on representation types of special biserial algebras.

In Section \ref{galois-covering-general}, we first recall the definitions of Galois covering of quivers with relations and the associated graded category (resp. the associated graded algebra) with the radical filtration of locally bounded category (resp. of finite-dimensional algebra), and then we show some general results on Galois covering of quivers with relations. As a consequence we show that for two special biserial algebras $A$ and $B$, if there exists a Galois covering between them, then $A$ and $B$ (resp. $\gr(A)$ and $\gr(B)$) have the same representation type.

In Section \ref{bfbgas}, we first recall the definition, the properties and examples of ${\bf bFBGA}$s, and then we recall and reinterpret the description on representation type of ${\bf bFBGA}$s in \cite{LL3}.

In Section \ref{sec-rep-type-graded-bFBGA}, we first recall some covering theory of ${\bf bFBGA}$s in \cite{LL2}. Based on this covering theory together with  our results in previous sections, we show that a graded ${\bf bFBGA}$ has the same representation type as the graded ${\bf BGA}$ of its reduced form.

In Section \ref{sec-quotient}, we describe the representation-finite and representation-infinite domestic $\gr(A)$ in terms of the quotient $\Ga/\ang{\nu}$ where $\Ga$ is the defining ${\bf bFBG}$ of the ${\bf bFBGA}$ $A$.

In Appendix \ref{appendix}, we describe the defining ${\bf bFBG}$ of a ${\bf bFBGA}$ whose associated graded algebra is representation-finite or representation-infinite domestic, and we also determine the number $n$ if the associated graded algebra is $n$-domestic in most cases.

\section*{Data availability} The datasets generated during the current study are available from the corresponding author on reasonable request.

\section{A general result on representation type of special biserial algebras}\label{rep.type-sba}

Throughout this paper all algebras will be finite-dimensional with $1$ over an algebraically closed field $\bk$, and all categories are $\bk$-linear and Hom-finite. Given an algebra $A$, we denote by $\rad(A)$ the Jacobson radical of $A$. We denote by $A$-mod the category of finite-dimensional left $A$-modules. For a module $M$ in $A$-mod, we denote by $\soc(M)$ and $\rad(M)$ the socle and the radical of $M$, respectively. The length of the module $M$ is denoted by $\ell(M)$ and it means the number of composition factors in any composition series of $M$.

In this paper, for a quiver $Q$ we denote by $Q_0$ and $Q_1$ its vertex set and arrow set respectively. As usual, the trivial path at a vertex $i$ is denoted by $e_i$. We write a path $p$ in a quiver from right to left and denote by $s(p)$ and $t(p)$ the start and the end of $p$, respectively. We define the length of a path in a quiver in an obvious way. For an algebra $A=\bk Q/I$ given by quiver with relations, by abuse of notation we sometimes view an element in $\bk Q$ as an element in the quotient $A=\bk Q/I$ if no confusion can arise.

\subsection{Special biserial algebras and string algebras}\label{subsec-sba}
\

In this subsection, we recall the definitions of special biserial algebras and string algebras (cf. \cite[Section II]{E}).

\begin{Def}  \label{special-biserial-algebra} A finite-dimensional $\bk$-algebra $A$ is called special biserial if there is a
	quiver $Q$ and an admissible ideal $I$ in $\bk Q$ such that $A$ is isomorphic to $\bk Q/I$ and
	such that $\bk Q/I$ satisfies the following conditions:\\
	$(1)$ At every vertex $v$ in $Q$ there are at most two arrows starting at $v$ and there are at most
	two arrows ending at $v$.\\
	$(2)$ For every arrow $\alpha$ in $Q$, there exists at most one arrow $\beta$ such that $\beta\alpha\notin I$ and there
	exists at most one arrow $\gamma$ such that $\alpha\gamma\notin I$.
	
	A special biserial algebra $A=\bk Q/I$ is called a string algebra if the defining ideal $I$ is generated by paths.
\end{Def}

Given a special biserial algebra $A=\bk Q/I$, we can associate a string algebra $\bar{A}$ as follows. Set
$$\mathbb{L}:=\{i\in Q_{0}\mid Ae_i\mbox{ is an injective and not uniserial module}\},$$
$$S_{0}:=\bigoplus\limits_{i\in \mathbb{L}}\soc(Ae_i),$$
where $Ae_i$ denotes the indecomposable projective $A$-module at vertex $i$. Then $S_{0}$ is an ideal of $A$ and the quotient algebra $\bar{A}=A/S_{0}$ is a string algebra. Note that  the operation $\overline{(\cdot)}$ preserves representation type 
and we can reconstruct the Auslander-Reiten quiver of $A$ from the  Auslander-Reiten quiver of $\overline{A}$ easily  (cf. \cite[Section II.1.3]{E}).

Suppose now that $A=\bk Q/I$ is a string algebra. For an arrow $\beta\in Q_{1}$, we denote by $\beta^{-1}$ the formal inverse of $\beta$ and set $s(\beta^{-1})=t(\beta)$, $t(\beta^{-1})=s(\beta)$, $(\beta^{-1})^{-1}=\beta$. For convenience, the formal inverse of an arrow will be called an inverse arrow. A word of length $n$ is defined by a sequence $c_{n}\ldots c_{2}c_{1}$, where $c_{i}\in Q_{1}$ or $c_{i}^{-1}\in Q_{1}$, and where $t(c_{i})=s(c_{i+1})$ for $1\leq i\leq n-1$. We define
$$s(c_{n}\ldots c_{2}c_{1})=s(c_{1}),~t(c_{n}\ldots c_{2}c_{1})=t(c_{n}),$$
and
$$(c_{n}\ldots c_{2}c_{1})^{-1}=c_{1}^{-1}c_{2}^{-1}\ldots c_{n}^{-1}.$$
For every vertex $v$ in $Q$, there is an empty word $1_{v}$ of length $0$ such that $t(1_{v})=s(1_{v})=v$ and $1_{v}^{-1}=1_{v}$.
A word $C$ is called closed if $s(C)=t(C)$. For a closed word $C=c_{n}\ldots c_{2}c_{1}$,
we define a rotation of $C$ as a word of the form $c_{i}\ldots c_{1}c_{n}\ldots c_{i+1}$. The product of two words is defined by placing them next to each other, provided that the resulting sequence is a word.

A word  $C$ is called a {\it string} provided either $C=1_{v}$ for some vertex $v$ in $Q$ or $C=c_{n}\ldots c_{2}c_{1}$ satisfying $c_{i+1}\neq c_{i}^{-1}$ for $1\leq i\leq n-1$, and no subword (or its inverse) of $C$ belongs to the ideal $I$. We say that a string $C=c_{n}\ldots c_{2}c_{1}$ with $n\geq 1$ is directed if all $c_i$ are arrows, and $C$ is inverse if all $c_i$ are inverse arrows. A closed string $C$ of positive length is called a {\it cycle} if all powers of $C$ are strings. Note that a cycle must contain both arrows and inverse arrows. A cycle $C$ which is not a power of a string of smaller length is said to be a {\it band}.

We consider two equivalence relations.~Firstly, $\sim$ denotes the relation on the set of strings which identifies $C$ and $C^{-1}$; and secondly, we define $\sim_{A}$ to be the equivalence relation on the set of cycles which identifies each cycle with its rotations and their inverses. Note that $\sim_{A}$ restricts to an equivalence relation on the set of bands. Let $\mathrm{St}(A)$ (or simply $\mathrm{St}$) be a set of representatives of strings in $A$ under $\sim$, and let $\mathrm{Ba}(A)$ (or simply $\mathrm{Ba}$) be the set of representatives of bands under $\sim_{A}$.

\subsection{Representation type}\label{subsec-rep-type-of-sba}
\

The main objective of this subsection is to prove Theorem \ref{Theorem}. Note that part of the results in this theorem were stated in \cite[Proposition 2.8]{GLY}, but the proof there is not complete.

Denote by $\bk[t]$ the polynomial algebra in one indeterminate $t$. We first recall the various notions on representation type of an algebra (cf. \cite[Section XIX.3]{SS}).

\begin{Def}\label{rep-type}
\

\begin{enumerate}[$\bullet$]
\item An algebra $A$ is said to be of tame representation type (or shortly tame) if for each positive integer $d$, there exist $A$-$\bk[t]$-bimodules $M_1,\cdots,M_{m_d}$ that are finitely generated and free as right $\bk[t]$-modules, such that all but a finite number of the isomorphism classes of indecomposable $A$-modules of dimension $d$ are isomorphic to modules of the form $M_i\otimes_{\bk[t]}\bk[t]/(t-\lambda)$, where $i\in\{1,2,\cdots,m_d\}$ and $\lambda\in \bk$.
\item Suppose that $A$ is tame. For each positive integer $d$, denote by $\mu_{A}(d)$ the least number of such $A$-$\bk[t]$-bimodules. $A$ is said to be of polynomial growth if there exists an integer $m\geq 1$ such that $\mu_{A}(d)\leq d^m$ for each $d\geq 2$.
\item $A$ is said to be domestic if there exist a finite number of $A$-$\bk[t]$-bimodules $M_1,\cdots,M_n$ that are finitely generated and free as right $\bk[t]$-modules, such that for each positive integer $d$, all but a finite number of the isomorphism classes of indecomposable $A$-modules of dimension $d$ are isomorphic to modules of the form $M_i\otimes_{\bk[t]}V$, where $i\in\{1,2,\cdots,n\}$ and $V$ is a finite-dimensional indecomposable $\bk[t]$-module. If $N$ is the least number of such $A$-$\bk[t]$-bimodules, then $A$ is called $N$-domestic.
\end{enumerate}
\end{Def}

Note that by the second Brauer-Thrall conjecture, algebras of finite representation type are precisely the $0$-domestic algebras. According to \cite[Corollary 5.7]{C}, an algebra $A$ is domestic if and only if there exists some positive integer $n$ such that $\mu_{A}(d)\leq n$ for all positive integers $d$, in particular, if an algebra $A$ is domestic then it is of polynomial growth (cf. \cite[Section XIX.3]{SS}). Moreover, if $A$ is $N$-domestic, then $N$ is the least number such that $\mu_{A}(d)\leq N$ for all positive integers $d$. In particular, every domestic algebra is tame.

Let $A=\bk Q/I$ be a special biserial algebra with $(Q,I)$ satisfying conditions $(1), (2)$ in Definition \ref{special-biserial-algebra}. Let $\overline{A}=A/S_0$ be the associated string algebra of $A$, which has the same representation type as $A$ (here $A$ is $n$-domestic if and only if $\overline{A}$ is $n$-domestic). Therefore to study the representation type of $A$, we may assume that $A$ is a string algebra.

In \cite{ES}, Erdmann and Skowro\'nski proved that for a self-injective special biserial algebra $A=\bk Q/I$, $A$ is representation-infinite domestic if and only if $A$ is representation-infinite of polynomial growth if and only if the cardinality of $\mathrm{Ba}(\overline{A})$ is a positive integer, where $\overline{A}$ is the associated string algebra of $A$. We generalize this statement to arbitrary special biserial algebras.

\begin{Thm}\label{Theorem}
Let $A=\bk Q/I$ be a special biserial algebra. The following are equivalent:
\begin{enumerate}[$(i)$]
\item $A$ is representation-infinite domestic;
\item $A$ is representation-infinite of polynomial growth;
\item The cardinality of $\mathrm{Ba}(\overline{A})$ is a positive integer, where $\overline{A}$ is the associated string algebra of $A$.
\end{enumerate}
Moreover, if the cardinality of $\mathrm{Ba}(\overline{A})$ is $n$, then $A$ is $n$-domestic.
\end{Thm}

\medskip
The following lemma is a remark in \cite[Subsection 2.1]{Sk}. For the convenience of the reader, we include a proof here.

 \begin{Lem} \label{constructible}
Let $A$ be a finite-dimensional $\bk$-algebra and let $M_1,M_2$ be $A$-$\bk[t]$-bimodules that are finitely generated free as right $\bk[t]$-modules. For any two positive integers $n_1$ and $n_2$, the set $X=\{\lambda_1\in \bk\mid M_1\otimes_{\bk[t]}\bk[t]/(t-\lambda_1)^{n_1}\cong M_2\otimes_{\bk[t]}\bk[t]/(t-\lambda_2)^{n_2}$ for some $\lambda_2\in \bk^{*}\}$ is a constructible subset of $\bk$, and hence finite or cofinite in $\bk$.
\end{Lem}

\begin{proof}
{\it Step 1: Assume that $A$ is basic}

Suppose that $A$ is isomorphic to $\bk Q/I$ for some quiver with relations $(Q,I)$. For each dimension vector ${\bf d}=(d_x)_{x\in Q_0}$ of $A$, denote by Mod$(A,{\bf d})$ the closed subset of
$$\prod_{\alpha\in Q_1}\mathrm{Hom}_{\bk}(\bk^{d_{s(\alpha)}},\bk^{d_{t(\alpha)}})$$
given by the relations in $I$ (here $s(\alpha)$ and $t(\alpha)$ denote the start and the end of the arrow $\alpha$ respectively), and let GL$({\bf d})=\prod_{x\in Q_0}\mathrm{GL}(d_x,\bk)$. Each element $\xi=(\xi_{\alpha})_{\alpha\in Q_1}\in$Mod$(A,{\bf d})$ gives a representation $R(\xi)$ of $A$ of dimension vector ${\bf d}$. Moreover, GL$({\bf d})$ acts on Mod$(A,{\bf d})$ by conjugation, and for $\xi,\eta\in$Mod$(A,{\bf d})$, $R(\xi)$ is isomorphic to $R(\eta)$ if and only if $\xi,\eta$ belong to the same GL$({\bf d})$-orbit.

For each $x\in Q_0$, denote by $d_{i,x}$ the rank of $e_x M_i$ as a  right $\bk[t]$-module, where $i=1,2$ and $e_x$ is the idempotent of $\bk Q/I$ corresponding to $x$. We may assume that $n_1 d_{1,x}=n_2 d_{2,x}$ for all $x\in Q_0$, otherwise $X=\varnothing$. Let $d_x=n_1 d_{1,x}$ for each $x\in Q_0$ and ${\bf d}=(d_x)_{x\in Q_0}$ be a dimension vector of $A$. For each $i\in\{1,2\}$, $M_i$ is isomorphic to some (may be infinite-dimensional) representation $L_i$ of $\bk Q/I$, where for each $x\in Q_0$, $(L_i)_x=\bk[t]^{d_{i,x}}$, and for each $\alpha\in Q_1$, $(L_i)_{\alpha}$ is given by the matrix $(f_{pq}^{(i)})_{d_{i,t(\alpha)}\times d_{i,s(\alpha)}}$ with elements in $\bk[t]$. Then for each $\lambda_i\in\bk$, $M_i\otimes_{\bk[t]}\bk[t]/(t-\lambda_i)^{n_i}$ is isomorphic to the representation $V_{i,\lambda_i}$ of $\bk Q/I$, where for each $x\in Q_0$, $(V_{i,\lambda_i})_x=\bk^{d_x}$, and for each $\alpha\in Q_1$, $(V_{i,\lambda_i})_{\alpha}$ is given by the block matrix $(f_{pq}^{(i)}(J_{n_i}(\lambda_i)))_{d_{i,t(\alpha)}\times d_{i,s(\alpha)}}$, where $J_{n_i}(\lambda_i)$ is the matrix
$$\begin{pmatrix}
\lambda_i & & & \\
1 & \lambda_i & & \\
& \ddots & \ddots & \\
& & 1 & \lambda_i
\end{pmatrix}_{n_i\times n_i}.$$
By considering $V_{i,\lambda_i}$ as an element in Mod$(A,{\bf d})$, we see that the map $\phi_i:\mathbb{A}^1\rightarrow$Mod$(A,{\bf d})$, $\lambda_i\mapsto V_{i,\lambda_i}$ is a morphism of varieties.

Consider the morphism of varieties
$$\psi:\mathbb{A}^1\times\mathrm{GL}({\bf d})\times(\mathbb{A}^1-\{0\})\rightarrow\mathrm{Mod}(A,{\bf d})\times\mathrm{Mod}(A,{\bf d}), \quad (\lambda_1,g,\lambda_2)\mapsto (\phi_{1}(\lambda_1),g\cdot\phi_{2}(\lambda_2)).$$
Denote by $\pi$ the projection
$$\mathbb{A}^1\times\mathrm{GL}({\bf d})\times(\mathbb{A}^1-\{0\})\rightarrow\mathbb{A}^1, \quad (\lambda_1,g,\lambda_2)\mapsto\lambda_1.$$
Then $X=\pi(\psi^{-1}(\Delta))$, where $\Delta$ is the closed subset $\{(\xi,\xi)\mid\xi\in\mathrm{Mod}(A,{\bf d})\}$ of $\mathrm{Mod}(A,{\bf d})\times\mathrm{Mod}(A,{\bf d})$. By Chevalley's Theorem (see \cite[Exercises 1.9.6.(1)]{Sp}), $X$ is a constructible subset of $\bk$.

\medskip
{\it Step 2: Assume that $A$ is not basic.}

Let $B$ be the basic algebra of $A$, so there exist a left-right projective $B$-$A$-bimodule $P$ and a left-right projective $A$-$B$-bimodule $Q$ such that $P\otimes_{A}Q\cong B$ as $B$-$B$-bimodules and $Q\otimes_{B}P\cong A$ as $A$-$A$-bimodules. Then
$$X=\{\lambda_1\in\bk\mid P\otimes_{A}M_1\otimes_{\bk[t]}\bk[t]/(t-\lambda_1)^{n_1}\cong P\otimes_{A}M_2\otimes_{\bk[t]}\bk[t]/(t-\lambda_2)^{n_2} \text{ for some } \lambda_2\in\bk^{*}\},$$
where $P\otimes_{A}M_1$, $P\otimes_{A}M_2$ are $B$-$\bk[t]$-bimodules that are finitely generated free as right $\bk[t]$-modules. Then by Step 1 we see that $X$ is a constructible subset of $\bk$.
\end{proof}

In the following of this subsection we always assume that $A=\bk Q/I$ is a string algebra unless otherwise stated. For each band $b$ in $A$ we construct an $A$-$\bk[t]$-bimodule $M_b$ which is finitely generated free as a right $\bk[t]$-module as follows: Let $b=\alpha_r\cdots\alpha_2\alpha_1$ with $\alpha_i\in Q_1$ or $\alpha_{i}^{-1}\in Q_1$ for each $i\in\mathbb{Z}/r\mathbb{Z}=\{1,2,\cdots,r\}$. For each vertex $x$ of $Q$, define
$$(M_b)_x=\bigoplus\limits_{s(\alpha_i)=x}V_i,$$
where $V_i=\bk[t]$ for each $i\in\mathbb{Z}/r\mathbb{Z}$ and $s(\alpha_i)$ is the start of $\alpha_i$, and for each arrow $\alpha$ of $Q$, define
\begin{equation*}
(M_{b})_{\alpha}|_{V_i}=\begin{cases}
V_i\xrightarrow{id}V_{i+1}, \text{ if } \alpha=\alpha_i \text{ and } i\neq r; \\
V_{i}\xrightarrow{id}V_{i-1}, \text{ if } \alpha=\alpha_{i-1}^{-1} \text{ and } i\neq 1; \\
V_r\xrightarrow{t}V_{1}, \text{ if } \alpha=\alpha_r \text{ and } i=r; \\
V_1\xrightarrow{t}V_{r}, \text{ if } \alpha=\alpha_r^{-1} \text{ and } i=1; \\
0, \text{ otherwise}.
\end{cases}
\end{equation*}
Note that for each positive integer $n$ and any $\lambda\in\bk^{*}$, $M_b\otimes_{\bk[t]}\bk[t]/(t-\lambda)^n$ is just the band module $M(b,n,\lambda)$ given by the band $b$ and the finite-dimensional indecomposable $\bk[t,t^{-1}]$-module $\bk[t]/(t-\lambda)^n$ (see \cite[Section II.3]{E}). This also shows that each band module over $A$ has this form.

\medskip

Let $c$ be a cycle in $A$. Then $c$ can be written uniquely as the form $b(c)^{m(c)}$, where $b(c)$ is a band in $A$ and $m(c)$ is a positive integer (it suffices to take $m(c)$ as the maximal positive integer with $c=b^{m(c)}$ for some subword $b$ of $c$, and take $b(c)$ as $b$).

\begin{Lem}\label{equivalent-cycle}
Two cycles $c_1$, $c_2$ in $A$ are equivalent if and only if the associated bands $b(c_1)$, $b(c_2)$ are equivalent and $m(c_1)=m(c_2)$.
\end{Lem}

\begin{proof}
''$\Rightarrow$'' We first assume that $c_2$ is obtained from $c_1$ via a rotation. Let $c_1=\alpha_r\cdots\alpha_2\alpha_1$ and $c_2=\alpha_1\alpha_r\cdots\alpha_2$, where $\alpha_i\in Q_1$ or $\alpha_{i}^{-1}\in Q_1$ for each $i$. Suppose that $b(c_1)=\beta_l\cdots\beta_2\beta_1$ with $\beta_i\in Q_1$ or $\beta_{i}^{-1}\in Q_1$ for each $i$, then it is straightforward to show that $c_2=(\beta_1\beta_l\cdots\beta_2)^{m(c_1)}$. Since $\beta_1\beta_l\cdots\beta_2$ is a band in $A$, we have $b(c_2)=\beta_1\beta_l\cdots\beta_2$ and $m(c_2)=m(c_1)$.

Now we assume that $c_2$ is the inverse of $c_1$. Since $c_1=b(c_1)^{m(c_1)}$, $c_2=c_{1}^{-1}=(b(c_1)^{-1})^{m(c_1)}$. Since $b(c_1)^{-1}$ is a band in $A$, we have $b(c_2)=b(c_1)^{-1}$ and $m(c_2)=m(c_1)$.

''$\Leftarrow$'' If $b(c_2)$ is obtained from $b(c_1)$ via a rotation (resp. an inverse), since $m(c_1)=m(c_2)$, $c_2$ is also obtained from $c_1$ via a rotation (resp. an inverse).
\end{proof}

For each cycle $c$ in $A$, we can associate a family of band modules $\mathscr{B}_c$ over $A$ as follows: $$\mathscr{B}_c=\{M(b(c),m(c),\lambda)\mid\lambda\in\bk^{*}\}.$$
Here $M(b(c),m(c),\lambda)$ is the band module given by the band $b(c)$ and the finite-dimensional indecomposable $\bk[t,t^{-1}]$-module $\bk[t]/(t-\lambda)^{m(c)}$, which is isomorphic to $M_{b(c)}\otimes_{\bk[t]}\bk[t]/(t-\lambda)^{m(c)}$. By Lemma \ref{equivalent-cycle}, two cycles $c_1$, $c_2$ are equivalent if and only if the they are associated to the same family of band modules.

\begin{Lem}\label{number-of-band-and-bimodules}
If there are $n$ pairwise non-equivalent cycles $c_1,c_2,\cdots,c_n$ of the same length $d$ in $A$, then $\mu_{A}(d)\geq n$. Especially, if the cardinality of $\mathrm{Ba}(A)$ is $n$, then $\mu_{A}(d)\geq n$ for some positive integer $d$.
\end{Lem}

\begin{proof}
Suppose that $\mu_{A}(d)=l<n$. Then there exist $A$-$\bk[t]$-bimodules $M_1,\cdots,M_{l}$ that are finitely generated and free as right $\bk[t]$-modules, such that all but a finite number of the isomorphism classes of indecomposable $A$-modules of dimension $d$ are isomorphic to modules of the form $M_i\otimes_{\bk[t]}\bk[t]/(t-\lambda)$, where $i\in\{1,2,\cdots,l\}$ and $\lambda\in\bk$. Let
$$X_{i,j}=\{\lambda\in\bk\mid M_i\otimes_{\bk[t]}\bk[t]/(t-\lambda)\cong M_{b(c_j)}\otimes_{\bk[t]}\bk[t]/(t-\mu)^{m(c_j)} \text{ for some }\mu\in\bk^{*} \}$$
for all $i\in\{1,2,\cdots,l\}$ and $j\in\{1,2,\cdots,n\}$. According to Lemma \ref{constructible}, each $X_{i,j}$ is a constructible subset of $\bk$, and hence finite or cofinite in $\bk$. For each $j\in\{1,2,\cdots,n\}$, since all but a finite number of modules in $\mathscr{B}_{c_j}=\{M_{b(c_j)}\otimes_{\bk[t]}\bk[t]/(t-\mu)^{m(c_j)}\mid\mu\in\bk^{*}\}$ are isomorphic to modules of the form $M_i\otimes_{\bk[t]}\bk[t]/(t-\lambda)$, where $i\in\{1,2,\cdots,l\}$ and $\lambda\in\bk$, there exists some $i(j)\in\{1,2,\cdots,l\}$ such that $X_{i(j),j}$ is cofinite in $\bk$. Since $l<n$, there exist some $1\leq j_1<j_2\leq n$ with $i(j_1)=i(j_2)$. Let $e=i(j_1)=i(j_2)$. Then both $X_{e,j_1}$ and $X_{e,j_2}$ are cofinite in $\bk$, and therefore $X_{e,j_1}\cap X_{e,j_2}\neq\varnothing$. If $\lambda\in X_{e,j_1}\cap X_{e,j_2}$, then there exist $\mu_1,\mu_2\in\bk^{*}$ such that
$$M_{b(c_{j_1})}\otimes_{\bk[t]}\bk[t]/(t-\mu_1)^{m(c_{j_1})}\cong M_e\otimes_{\bk[t]}\bk[t]/(t-\lambda)\cong M_{b(c_{j_2})}\otimes_{\bk[t]}\bk[t]/(t-\mu_2)^{m(c_{j_2})}.$$
Therefore $b(c_{j_1})$ and $b(c_{j_2})$ are equivalent bands and $m(c_{j_1})=m(c_{j_2})$. By Lemma \ref{equivalent-cycle}, $c_{j_1}$ and $c_{j_2}$ are equivalent cycles, a contradiction.
\end{proof}

Part of the following result was stated in \cite[Proposition 2.8]{GLY} but the proof there is not complete, here we give a proof based on Lemma \ref{number-of-band-and-bimodules}.

\begin{Prop}\label{number-of-band-and-representation-type}
Let $A=\bk Q/I$ be a string algebra. Then $A$ is representation-infinite domestic if and only if the cardinality of $\mathrm{Ba}(A)$ is a positive integer. Moreover, if the cardinality of $\mathrm{Ba}(A)$ is $n$, then $A$ is $n$-domestic.
\end{Prop}

\begin{proof}
''$\Leftarrow$'' Suppose that $A=\bk Q/I$ is a string algebra such that $\mathrm{Ba}(A)=\{b_1,b_2,\cdots,b_n\}$.
Then each band module over $A$ is isomorphic to a module of the form $M_{b_i}\otimes_{\bk[t]}V$, where $1\leq i\leq n$ and $V$ is a finite-dimensional indecomposable $\bk[t]$-module. Since for each positive integer $d$, the number of isomorphism classes of string modules over $A$ of dimension $d$ is finite, $A$ is $m$-domestic, where $m\leq n$. Since the cardinality of $\mathrm{Ba}(A)$ is $n$, by Lemma \ref{number-of-band-and-bimodules}, there exists some positive integer $d$ with $\mu_{A}(d)\geq n$. Since $m\geq\mu_{A}(e)$ for all positive integers $e$, we imply that $m=n$.

''$\Rightarrow$'' Since $A$ is representation-infinite, there exists a band in $\mathrm{Ba}(A)$. If there are infinitely many different elements in $\mathrm{Ba}(A)$, then according to Lemma \ref{number-of-band-and-bimodules} there exists no positive integer $n$ with $\mu_{A}(d)\leq n$ for all positive integers $d$. Therefore $A$ is nondomestic.
\end{proof}

\begin{Lem}\label{not-of-polynomial-growth}
If $c_1$, $c_2$ are two different cycles in $A$ at the same vertex of $Q$ of the same length $l$ such that $c_2 c_1$ is again a cycle, then $A$ is not of polynomial growth.
\end{Lem}

\begin{proof}
Let $\mathscr{C}_n=\{c_{i(n)}\cdots c_{i(2)}c_{i(1)}\mid i(1),\cdots,i(n)\in\{1,2\}\}$ be a set of cycles in $A$ for each positive integer $n$. For each cycle $c\in\mathscr{C}_n$, since the length of $c$ is $nl$, there are at most $2nl$ cycles in $A$ which are equivalent to $c$. Therefore there are at least $\frac{2^n}{2nl}=\frac{2^{n-1}}{nl}$ cycles in $\mathscr{C}_n$ up to equivalence. According to Lemma \ref{number-of-band-and-bimodules}, $\mu_{A}(nl)\geq\frac{2^{n-1}}{nl}$ for all positive integers $n$, then $A$ is not of polynomial growth.
\end{proof}

\begin{proof}[{\bf Proof of Theorem \ref{Theorem}}]
The equivalence of $(i)$ and $(iii)$ is a consequence of Proposition \ref{number-of-band-and-representation-type} and the implication $(i)\Rightarrow (ii)$ is well-known.

''$(ii)\Rightarrow (iii)$'' Suppose that there are infinitely many different elements in $\mathrm{Ba}(\overline{A})$, where $\overline{A}$ is the associated string algebra of $A$. For each band $b=\alpha_r\cdots\alpha_2\alpha_1$ in $\overline{A}$, where $\alpha_i\in Q_1$ or $\alpha_{i}^{-1}\in Q_1$ for each $i\in\mathbb{Z}/r\mathbb{Z}=\{1,2,\cdots,r\}$, we denote by top$(b)$ the subset of $Q_0$ formed by vertices of the form $s(\alpha_i)$ for some $i\in\mathbb{Z}/r\mathbb{Z}=\{1,2,\cdots,r\}$ with $\alpha_i\in Q_1$ and $\alpha_{i-1}^{-1}\in Q_{1}$ (note that equivalent bands have the same tops). Since there are infinitely many different elements in $\mathrm{Ba}(\overline{A})$, there exist two different bands $b_1$ and $b_2$ in $\mathrm{Ba}(\overline{A})$ with $\mathrm{top}(b_1)\cap\mathrm{top}(b_2)\neq\varnothing$. Fix $x\in$top$(b_1)\cap$top$(b_2)$, and after rotations we may assume that $b_1=\alpha_r\cdots\alpha_2\alpha_1$ and $b_2=\beta_l\cdots\beta_2\beta_1$ are bands starting and ending at $x$, where $\alpha_1,\beta_1\in Q_1$ and $\alpha_{r}^{-1},\beta_{l}^{-1}\in Q_{1}$.

Since $b_1$ and $b_2$ are bands, we have $\alpha_r\neq\alpha_{1}^{-1}$ and $\beta_l\neq\beta_{1}^{-1}$. If $\alpha_1\neq\beta_{l}^{-1}$ and $\alpha_r\neq\beta_{1}^{-1}$, then $b_1 b_2$ is a cycle in $\overline{A}$. If $\alpha_1=\beta_{l}^{-1}$, then $\alpha_r\neq \beta_l$ and $\alpha_1\neq\beta_1$, so $b_{1}^{-1}b_2=\alpha_{1}^{-1}\cdots\alpha_{r}^{-1}\beta_l\cdots\beta_1$ is a cycle in $\overline{A}$. If $\alpha_r=\beta_{1}^{-1}$, then we also have $\alpha_r\neq \beta_l$ and $\alpha_1\neq\beta_1$, so $b_{1}^{-1}b_2=\alpha_{1}^{-1}\cdots\alpha_{r}^{-1}\beta_l\cdots\beta_1$ is a cycle in $\overline{A}$. Therefore we may assume that $b_1 b_2$ is a cycle in $\overline{A}$ by replacing $b_1$ with $b_{1}^{-1}$ if necessary.

Let $c_1=b_{1}^{l}$ and $c_2=b_{2}^{r}$. Then $c_1$, $c_2$ are two different cycles in $\overline{A}$ at the same vertex $x$ of $Q$ of the same length such that $c_2 c_1$ is again a cycle. According to Lemma \ref{not-of-polynomial-growth}, $\overline{A}$ is not of polynomial growth, a contradiction.
\end{proof}

\section{Some general results on Galois covering of quivers with relations}\label{galois-covering-general}

In this section we will show some general results (see Propositions \ref{prop-covering-on-graded-categories}, \ref{prop-induced-covering}) that a Galois covering between two finite-dimensional algebras induces Galois coverings between the associated graded algebras with the radical filtration and the associated quotient algebras by the ideals of the sum of socles of some non-uniserial indecomposable projective-injective modules. As a corollary we show that for two special biserial algebras $A$ and $B$, if there exists a Galois covering between them, then $A$ and $B$ (resp. $\gr(A)$ and $\gr(B)$) have the same representation type.

We first recall (Galois) coverings of quivers with relations. For more details on covering theory, we refer to \cite{BG,Ga1981,MP,E,LL2,LLXC}.

\begin{Def}[{see \cite[Section I.10.1]{E}}] Let $Q$ and  $Q'$ be quivers,  and  $f:Q\rightarrow Q'$ a quiver morphism. $f$ is called a covering of quivers if for each vertex $x$ in $Q$, $f$ induces a bijection between the set of arrows of $Q$ starting (resp. ending) at $x$ and the set of arrows of $Q'$ starting (resp. ending) at $f(x)$. If $f$ is a surjective covering of quivers and there exists a group $G$ of automorphisms of $Q$ which acts freely on vertices of $Q$ (that is,  for each $g\in G$ and each $x\in Q_0$, $g(x)=x$ implies $g=1$) such that $fg=f$ for each $g\in G$ and $G$ acts transitively on each fiber $f^{-1}(y)$ of $f$ with $y\in Q'_0$, then $f$ is said to be a Galois covering of quivers (with group $G$).
\end{Def}

\begin{Rem}
Let $f:Q\rightarrow Q'$ be a covering of quivers.
\begin{enumerate}[$(1)$]
		\item For each arrow $\alpha'$ starting (resp. ending) at $f(x)$ for any vertex $x$ in $Q$, there exists a unique arrow $\alpha$ starting (resp. ending) at $x$ such that $f(\alpha)=\alpha'$. Moreover, we can add the definition of $f$ on paths in $Q$ as follows $f(\alpha_n\cdots\alpha_2\alpha_1):=f(\alpha_n)\cdots f(\alpha_2)f(\alpha_1)$, where $\alpha_i$ is an arrow in $Q$ for each $i=1,2,\cdots,n$. Therefore, $f$ has the unique lifting of paths property, that is, for any path $\rho'$ starting (resp. ending) at $f(x)$ in $\bk Q'$,  there exists a unique path $\rho$ starting (resp. ending) at $x$ such that $f(\rho)=\rho'$.
		\item If $Q'$ is a connected quiver and $b$ is a vertex in $Q'$, then there exists some vertex $x$ in $Q$ such that $f(x)=b$.
	\end{enumerate}

\end{Rem}

For a $\bk$-linear category $M$, we denote by $\mathrm{Obj}(M)$ the collection of objects of $M$ and by $M(x,y)$ the morphism set from object $x$ to object $y$, and we denote by $\rad M$ the radical of $M$, which is the ideal in $M$ defined by
$$\rad M(x,y)=\{f\in M(x,y)\mid 1_x-gf \text{ is invertible for each } g\in M(y,x)\}.$$
Inductively, for each positive integer $n>1$, define $\rad^n M$ as the ideal in $M$ given by
$$\rad^n M(x,y)=\sum_{z\in\mathrm{Obj}(M)}\rad^{n-1} M(z,y)\rad M(x,z).$$

Recall from \cite[Section 3.1]{BG} that a $k$-linear functor $F: M\rightarrow N$ between two $k$-categories is said to be a {\it covering functor} if for every two objects $a,b$ of $N$, $F$ induces isomorphisms $\bigoplus_{z/a}M(z,y)\rightarrow N(a,b)$ and $\bigoplus_{z/b}M(x,z)\rightarrow N(a,b)$, where $x,y$ are objects of $M$ with $a=F(x)$, $b=F(y)$. In particular, a covering functor is a faithful functor.

For a quiver morphism $f:Q\rightarrow Q'$, denote by $\bk f:\bk Q\rightarrow\bk Q'$ the $\bk$-linear functor between the associated path categories defined by the $\bk$-linear extension of $f$. It is well-known that if $f:Q\rightarrow Q'$ is a covering, then $\bk f:\bk Q\rightarrow\bk Q'$ is a covering functor.

\begin{Def}[{see \cite[Section 1]{MP}}]
Let $Q$ and $Q'$ be quivers and $I$ and $I'$ be ideals in path categories $\bk Q$ and $\bk Q'$ respectively. A Galois covering of quivers with relations $f:(Q,I)\rightarrow (Q',I')$ with group $G$ is a Galois covering of quivers $f:Q\rightarrow Q'$ with group $G$ which preserves relations such that $(\bk g)(I)=I$ for each $g\in G$ and such that $I'$ is the ideal generated by $(\bk f)(\rho)$ with $\rho\in I$.
\end{Def}

Galois coverings of quivers with relations are closely related to Galois covering functors between locally bounded categories. Recall from \cite[Definition 2.1]{BG} that a {\it locally bounded category} is a $k$-category $M$ satisfying the following conditions:
\begin{enumerate}[$(1)$]
  \item For each $x\in M$, the endomorphism algebra $M(x,x)$ is local.
  \item Distinct objects of $M$ are not isomorphic.
  \item For each $x\in M$, $\sum_{y\in M}\mathrm{dim}_{k}M(x,y)< \infty$ and $\sum_{y\in \Lambda}\mathrm{dim}_{k}M(y,x)< \infty$.
\end{enumerate}
A covering functor $F: M\rightarrow N$ between two locally bounded categories is called a {\it Galois covering} (see \cite[Section 3.1]{Ga1981}) if $F$ is surjective on objects and there exists a group $G$ of $k$-linear automorphisms of $M$ which acts freely on (objects of) $M$, such that $Fg=F$ for any $g\in G$ and $G$ acts transitively on $F^{-1}(a)$ for each object $a$ of $N$. It is well-known that a Galois covering of locally finite quivers with admissible relations $f:(Q,I)\rightarrow (Q',I')$ induces a Galois covering functor $F:\bk Q/I\rightarrow\bk Q'/I'$ between locally bounded categories.

For a locally bounded category $M$, denote by $\gr(M)$ the graded category associated with the radical filtration of $M$ (see \cite[Section 5.1]{BG}), where the objects in $\gr(M)$ are same as the objects in $M$ and
$$\gr(M)(x,y)=M(x,y)/\mathrm{rad}M(x,y)\oplus\mathrm{rad}M(x,y)/\mathrm{rad}^2 M(x,y)\oplus\mathrm{rad}^2 M(x,y)/\mathrm{rad}^3 M(x,y)\cdots$$
for each two objects $x$, $y$ in $\gr(M)$, and the composition of morphisms in $\gr(M)$ are induced by the composition of morphisms in $M$.

For a finite-dimensional algebra $A$ there is also the notion of associated graded algebra (denoted by $\mathrm{gr}(A)$) with the radical filtration of $A$, where $$\gr(A)=A/\mathrm{rad}(A)\oplus\mathrm{rad}(A)/\mathrm{rad}^2(A)\oplus\mathrm{rad}^2(A)/\mathrm{rad}^3(A)\cdots$$
and whose multiplication is induced by the multiplication of $A$. As a (geometric) degeneration of $A$, $\mathrm{gr}(A)$ has received much attention in representation theory (see \cite{CPS,Ga,Ge,RR}). Note also that if $A$ is a special biserial algebra, then so is $\gr(A)$ (cf. Lemma \ref{lem-graded-sba-is-sba} below).

The following lemma is straightforward. It shows that the graded category of a finite category and the graded algebra of its associated finite-dimensional algebra have the same quiver with relations.

\begin{Lem}	\label{lem-grA-grM}
If $M$ is a (Hom-finite) $k$-category with finitely many objects, denote by $$A=\bigoplus_{x,y\in\mathrm{Obj}(M)}M(x,y)$$
its associated (finite-dimensional) $k$-algebra, then the associated algebra of the graded category $\gr(M)$ is just the associated graded algebra with radical filtration $\gr(A)$ of $A$.
\end{Lem}

\begin{Lem}\label{lem-graded-sba-is-sba}
	If $A$ is a special biserial algebra, then $\gr(A)$ is also special biserial.
\end{Lem}

\begin{proof}
	Suppose that $A$ is given by quiver with relations $(Q,I)$ which satisfies conditions $(1)$, $(2)$ in Definition \ref{special-biserial-algebra}. Denote by $H:\bk Q\rightarrow \gr(A)$ the natural projection which maps each $e_i$ with $i\in Q_0$ to $e_i+\rad(A)\in A/\rad(A)$ and maps each arrow $\alpha$ in $Q$ to $\alpha+\rad^2(A)\in\rad(A)/\rad^2(A)$. Denote by $I'$ the kernel of $H$. It is straightforward to show that the quiver with relations $(Q,I')$ satisfies conditions $(1)$, $(2)$ in Definition \ref{special-biserial-algebra}.
\end{proof}

Recall from \cite[Section 3.2]{BG} that if $M$ and $N$ are two locally bounded categories and $F: M\rightarrow N$ is a covering functor, then $F$ induces a push-down functor $F_{\lambda}:\mathrm{mod}M\rightarrow\mathrm{mod}N$ between the associated categories of finitely generated modules.

Let $M$ and $N$ be locally bounded categories given by quivers with admissible ideals $(Q,I)$ and $(Q',I')$ respectively. Let $f:(Q,I)\rightarrow (Q',I')$ be a Galois covering of quivers with relations with group $G$. Denote by $F:M\rightarrow N$ the Galois covering functor induced by $f$, and let $F_{\lambda}:\mathrm{mod}M\rightarrow\mathrm{mod}N$ be the associated push-down functor. Since $F_{\lambda}$ preserves the radical of finitely generated modules (see \cite[Proposition 3.2]{BG}), there exists an isomorphism
$$F_{\lambda}(\mathrm{rad}^{n}M(-,x))\xrightarrow{\sim}\mathrm{rad}^{n}N(-,Fx)$$
for each $x\in M$ and each positive integer $n$. Then $F$ induces a bijection
$$\bigoplus_{Fy=a}\mathrm{rad}^{n}M(y,x)\xrightarrow{\sim}\mathrm{rad}^{n}N(a,Fx)$$
for each $x\in M$, $a\in N$ and each positive integer $n$. By using the covering functor $F:M^{\mathrm{op}}\rightarrow N^{\mathrm{op}}$, we have that $F$ also induces a bijection
$$\bigoplus_{Fy=a}\mathrm{rad}^{n}M(x,y)\xrightarrow{\sim}\mathrm{rad}^{n}N(Fx,a)$$
for each $x\in M$, $a\in N$ and each positive integer $n$. Then $F$ induces a covering functor $\gr(F):\gr(M)\rightarrow\gr(N)$, which is a Galois covering with group $G$. Denote by $H:\bk Q\rightarrow\gr(M)$ and $H':\bk Q'\rightarrow\gr(N)$ the natural projections, where the functor $H$ is given by $H(x)=x$ for each $x\in Q_0$ and $$H(\alpha)=\overline{\alpha}+\mathrm{rad}^{2}M(x,y)\in\mathrm{rad}M(x,y)/\mathrm{rad}^{2}M(x,y)$$
for each arrow $\alpha:x\rightarrow y$ in $Q$, where $\overline{\alpha}$ denotes the residue class of $\alpha$ in $M=\bk Q/I$, and $H'$ is defined in the same way. Denote by $L$ and $L'$ the kernels of $H$ and $H'$ respectively. Then $\gr(M)$ and $\gr(N)$ are given by the quivers with relations $(Q,L)$ and $(Q',L')$ respectively.

\begin{Prop}\label{prop-covering-on-graded-categories}
Following above notations, we have that the covering functor $\bk f:\bk Q\rightarrow \bk Q'$ maps each relation in $L$ to a relation in $L'$, and $f$ induces a Galois covering $f:(Q,L)\rightarrow (Q',L')$ with group $G$. Denote by $\phi:\bk Q/L\rightarrow\bk Q'/L'$ the Galois covering functor induced by $f$. Then there exists a commutative diagram
$$\xymatrix{
\bk Q/L\ar[r]^{\phi}\ar[d]_{h}^{\cong}&	\bk Q'/L'\ar[d]_{h'}^{\cong}	\\
\gr(M)\ar[r]^{\gr(F)}	& \gr(N),} $$
where $h$ and $h'$ are isomorphism functors induced by $H$ and $H'$ respectively.
\end{Prop}

\begin{proof}
It is straightforward to verify that the diagram
\begin{equation*}\label{diagram}
\xymatrix{
\bk Q\ar[r]^{\bk f}\ar[d]_{H}&	\bk Q'\ar[d]_{H'}	\\
\gr(M)\ar[r]^{\gr(F)}	& \gr(N)}
\end{equation*}
commutes. Then $\bk f$ sends relations in $L$ to $L'$. Since $\gr(F)$ is a covering functor, it is fully faithful. For each $g\in G$, by the commutative diagram above we have that
$$\gr(F)H=H'(\bk f)=H'(\bk f)(\bk g^{-1})=\gr(F)H(\bk g^{-1}).$$
Then
$$L=\mathrm{ker}(H)=\mathrm{ker}(\gr(F)H)=\mathrm{ker}(\gr(F)H(\bk g^{-1}))=\mathrm{ker}(H(\bk g^{-1}))=(\bk g)(L).$$
To show that $f$ induces a Galois covering $f:(Q,L)\rightarrow (Q',L')$ with group $G$, it suffices to show that $L'$ is generated by $(\bk f)(\rho)$ with $\rho\in L$. For each $\rho'\in L'(a,b)$ and each $y\in Q_0$ with $f(y)=b$, there exists a unique element
$$(\rho_x)\in\bigoplus_{f(x)=a}\bk Q(x,y)$$
such that
$$\sum_{f(x)=a}(\bk f)(\rho_x)=\rho'.$$
Since
$$\sum_{f(x)=a}\gr(F)(H(\rho_x))=\sum_{f(x)=a}H'((\bk f)(\rho_x))=H'(\sum_{f(x)=a}(\bk f)(\rho_x))=H'(\rho')=0$$
and since $\gr(F)$ induces a bijection
$$\bigoplus_{f(x)=a}\gr(M)(x,y)\xrightarrow{\sim}\gr(N)(a,b),$$
we have that $H(\rho_x)=0$ for each $x\in Q_0$ with $f(x)=a$. Then $\rho_x\in L(x,y)$ for each $x\in Q_0$ with $f(x)=a$, and therefore $\rho'$ is a sum of elements of the form $(\bk f)(\rho)$ with $\rho\in L$.

The commutativity of the diagram
$$\xymatrix{
\bk Q/L\ar[r]^{\phi}\ar[d]_{h}^{\cong}&	\bk Q'/L'\ar[d]_{h'}^{\cong}	\\
\gr(M)\ar[r]^{\gr(F)}	& \gr(N),} $$
follows directly from the commutativity of the diagram
$$\xymatrix{
\bk Q\ar[r]^{\bk f}\ar[d]_{H}&	\bk Q'\ar[d]_{H'}	\\
\gr(M)\ar[r]^{\gr(F)}	& \gr(N).} $$
\end{proof}

\begin{Prop}\label{prop-preserve-proj-inj}
Let $M$, $N$ be locally bounded categories. Let $F:M\rightarrow N$ be a covering functor and $F_{\lambda}:\mathrm{mod}M\rightarrow\mathrm{mod}N$ be the associated push-down functor. Then for each object $x$ in $M$, the module $M(-,x)$ is injective if the module $N(-,Fx)\cong F_{\lambda}M(-,x)$ is injective. Moreover, if $F$ is a Galois covering, then $M(-,x)$ is injective if and only if $N(-,Fx)\cong F_{\lambda}M(-,x)$ is injective.
\end{Prop}

\begin{proof}
Suppose that $N(-,Fx)\cong F_{\lambda}M(-,x)$ is injective. Since $F_{\lambda}$ is an exact functor mapping simples to simples, $F_{\lambda}(\mathrm{soc}(M(-,x)))$ is contained in $\mathrm{soc}(F_{\lambda}M(-,x))$. Since $F_{\lambda}M(-,x)$ is indecomposable and injective, $\mathrm{soc}(F_{\lambda}M(-,x))$ is simple. Then $F_{\lambda}(\mathrm{soc}(M(-,x)))$ is simple, which has dimensionality $1$. Since $F_{\lambda}$ preserves the dimensions of modules, $\mathrm{soc}(M(-,x))$ has dimensionality $1$, which is simple. Suppose that $\mathrm{soc}(M(-,x))\cong \bk_{y}$, where $\bk_y$ denotes the simple $M$-module corresponding to the object $y$ in $M$. Then $M(-,x)$ is a submodule of the injective hull $DM(y,-)$ of $\bk_y$. Since $\bk_y$ is a submodule of $M(-,x)$, $\bk_{Fy}=F_{\lambda}(\bk_y)$ is a submodule of $F_{\lambda}M(-,x)$. Since $F_{\lambda}M(-,x)$ is indecomposable injective, $\mathrm{soc}(F_{\lambda}M(-,x))\cong\bk_{Fy}$, and $F_{\lambda}M(-,x)\cong DN(Fy,-)$. Since
\begin{align*}
\mathrm{dim}_{\bk}(M(-,x))&=\mathrm{dim}_{\bk}(F_{\lambda}M(-,x))=\mathrm{dim}_{\bk}(DN(Fy,-))\\
&=\mathrm{dim}_{\bk}(N(Fy,-))=\mathrm{dim}_{\bk}(M(y,-))=\mathrm{dim}_{\bk}(DM(y,-)),
\end{align*}
$M(-,x)$ is isomorphic to $DM(y,-)$, which is injective.

Now suppose that $F$ is a Galois covering with group $G$ and $M(-,x)$ is injective. Then $M(-,x)$ is isomorphic to $DM(y,-)$ for some object $y$ in $M$. Let $S$ be a simple submodule of $N(-,Fx)$. Suppose that $S\cong \bk_a$ for some $a\in N$. Let $f\in N(a,Fx)$ be the morphism such that $N(-,f):N(-,a)\rightarrow N(-,Fx)$ is equal to the composition of the projective cover $p:N(-,a)\rightarrow S$ and the inclusion map $i:S\rightarrow N(-,Fx)$. Since $F$ is a covering functor, $F$ induces a bijection
$$\bigoplus_{Fz=a}M(z,x)\rightarrow N(a,Fx).$$
Suppose that $\sum_{Fz=a}F(u_z)=f$ with $u_z\in M(z,x)$ for each $z$. Fix an object $z_0$ of $M$ with $Fz_0=a$. Then for each object $t\in M$ and each morphism $h\in\mathrm{rad}M(t,z_0)$, since
$$fF(h)=\sum_{g\in G}F(u_{g(z_0)})F(g(h))=\sum_{g\in G}F(u_{g(z_0)}g(h)),$$
the bijection
$$\bigoplus_{g\in G}M(g(t),x)\rightarrow N(Ft,Fx)$$
maps $(u_{g(z_0)} g(h))_{g\in G}$ to $fF(h)$.
Since $h\in\mathrm{rad}M(t,z_0)$, we have $F(h)\in\mathrm{rad}N(Ft,a)$. Since $N(-,f)$ factors through a simple module, $\mathrm{rad}N(-,a)\subseteq\mathrm{ker}(N(-,f))$. Since $F(h)\in\mathrm{rad}N(Ft,a)$, $N(-,fF(h))=N(-,f)N(-,F(h))=0$ and $fF(h)=0$. Then $u_{g(z_0)} g(h)=0$ for each $g\in G$, in particular, $u_{z_0} h=0$. Since the composition of $M(-,h):M(-,t)\rightarrow M(-,z_0)$ and $M(-,u_{z_0}):M(-,z_0)\rightarrow M(-,x)$ is equal to $0$ for each object $t\in M$ and each morphism $h\in\mathrm{rad}M(t,z_0)$, $M(-,u_{z_0})$ factors through the top of $M(-,z_0)$. Then the image of $M(-,u_{z_0})$ belongs to $\mathrm{soc}(M(-,x))$. Since $z_0$ is arbitrary, the image of
$$(M(-,u_z))_{Fz=a}:\bigoplus_{Fz=a}M(-,z)\rightarrow M(-,x)$$
belongs to $\mathrm{soc}(M(-,x))$. By applying the push-down functor $F_{\lambda}$, we have the following commutative diagram
$$\xymatrix{
N(-,a)\ar[r]^(0.4){\iota} & \oplus_{Fz=a}N(-,a)\ar[rr]^{(N(-,F(u_z)))_{Fz=a}}\ar[d]^{\cong} & &	N(-,Fx)\ar[d]^{\cong} \\
& F_{\lambda}(\oplus_{Fz=a}M(-,z))\ar[rr]^{F_{\lambda}((M(-,u_z))_{Fz=a})}\ar[dr] &	& F_{\lambda}M(-,x) \\
&&F_{\lambda}(\mathrm{soc}M(-,x))\ar[ur]& ,} $$
where the component of $\iota$ corresponding to $z$ is equal to $1_{N(-,a)}$ if $u_z\neq 0$, and is equal to $0$ if $u_z=0$. Since the composition of the two maps in the first row of above diagram is $N(-,f)$, we have that $N(-,f)$ factors through the inclusion map $F_{\lambda}(\mathrm{soc}(M(-,x)))\hookrightarrow F_{\lambda}M(-,x)\cong N(-,Fx)$. Since the image of $N(-,f)$ is $S$, $S$ is contained in $F_{\lambda}(\mathrm{soc}(M(-,x)))$. Then each simple submodule of $N(-,Fx)$ is contained in $F_{\lambda}(\mathrm{soc}(M(-,x)))$, which implies that $\mathrm{soc}(N(-,Fx))=F_{\lambda}(\mathrm{soc}(M(-,x)))$. Since $M(-,x)$ is indecomposable injective, $\mathrm{soc}(M(-,x))$ is simple and therefore $\mathrm{soc}(N(-,Fx))=F_{\lambda}(\mathrm{soc}(M(-,x)))$ is simple. Since $\bk_y$ is a submodule of $M(-,x)\cong DM(y,-)$, $\bk_{Fy}=F_{\lambda}(\bk_y)$ is a submodule of $N(-,Fx)$. Since $\mathrm{soc}(N(-,Fx))$ is simple, it is isomorphic to $\bk_{Fy}$. Then $N(-,Fx)$ is isomorphic to a submodule of $DN(Fy,-)$. Since
\begin{align*}
\mathrm{dim}_{\bk}(N(-,Fx))&=\mathrm{dim}_{\bk}(M(-,x))=\mathrm{dim}_{\bk}(DM(y,-)) \\
&=\mathrm{dim}_{\bk}(M(y,-))=\mathrm{dim}_{\bk}(N(Fy,-))=\mathrm{dim}_{\bk}(DN(Fy,-)),	
\end{align*}
$N(-,Fx)$ is isomorphic to $DN(Fy,-)$, which is injective.
\end{proof}

Combining Proposition \ref{prop-preserve-proj-inj} and \cite[Proposition 3.2]{BG}, we have

\begin{Cor}\label{cor-projective-injective}
Let $M$, $N$ be locally bounded categories. Let $F:M\rightarrow N$ be a Galois covering functor and $F_{\lambda}:\mathrm{mod}M\rightarrow\mathrm{mod}N$ be the associated push-down functor. Then for each $m\in\mathrm{mod}M$, $m$ is projective-injective if and only if $F_{\lambda}m$ is projective-injective.
\end{Cor}

Let $A$ and $A'$ be two finite-dimensional algebras given by quivers $Q$ and $Q'$ with admissible ideals $I\subseteq \bk Q$ and $I'\subseteq \bk Q'$ respectively, and suppose that there exists a Galois covering $f:(Q,I)\rightarrow (Q',I')$ with group $G$. Set
$$\mathbb{L}:=\{i\in Q_{0}\mid Ae_i\mbox{ is an injective and not uniserial module}\}$$
and
$$\mathbb{L}':=\{i'\in Q'_{0}\mid A'e_{i'}\mbox{ is an injective and not uniserial module}\},$$
where $Ae_i$ (resp. $A'e_{i'}$) denotes the indecomposable projective $A$-module (resp. $A'$-module) at vertex $i$ (resp. $i'$). Since push-down functor preserves the radical filtration of a module, for $i\in Q_0$, the indecomposable projective module $Ae_i$ is uniserial if and only if the indecomposable projective module $A'e_{f(i)}$ is uniserial. Then it follows from Corollary \ref{cor-projective-injective} that a vertex $i\in Q_0$ belongs to $\mathbb{L}$ if and only if $f(i)$ belongs to $\mathbb{L}'$.

Let $T'$ be a subset of $\mathbb{L}'$ and let $T=f^{-1}(T')$ be a subset of $\mathbb{L}$. Set
$$S_{T}:=\bigoplus\limits_{i\in T}\soc(Ae_i)$$
and
$$S'_{T'}:=\bigoplus\limits_{i'\in T'}\soc(A'e_{i'}).$$
Then $S_{T}$ (resp. $S'_{T'}$) is an ideal of $A$ (resp. $A'$).
For each $a\in T'$, choose some $x_a\in f^{-1}(a)$ and choose a uniform relation $\rho_{x_a}$ in $\bk Q$ with $s(\rho_{x_a})=e_{x_a}$ such that the residue class of $\rho_{x_a}$ in $A$ belongs to $\soc(Ae_{x_a})-\{0\}$. For each $y\in T$, define $\rho_y=g(\rho_{x_{f(y)}})$, where $g$ is the unique element in $G$ with $g(x_{f(y)})=y$. Then $\rho_y$ is a uniform relation in $\bk Q$ with $s(\rho_{y})=e_{y}$ such that the residue class of $\rho_{y}$ in $A$ belongs to $\soc(Ae_{y})-\{0\}$. Moreover, for each $a\in T'$, define $\rho'_a=(\bk f)(\rho_{x_a})$, which is a uniform relation in $\bk Q'$ with $s(\rho'_{a})=e_{a}$. Since $f:(Q,I)\rightarrow (Q',I')$ is a Galois covering of quivers with relations, it is straightforward to show that the residue class of $\rho'_{a}$ in $A'$ belongs to $\soc(A'e_{a})-\{0\}$.

The quotient algebras $\overline{A}=A/S_T$ and $\overline{A'}=A'/S'_{T'}$ are given by quivers with relations $(Q,I_T)$ and $(Q',I'_{T'})$ respectively, where $I_T$ (resp. $I'_{T'}$) is the ideal in $\bk Q$ (resp. $\bk Q'$) generated by $I$ (resp. $I'$) together with relations $\rho_y$ with $y\in T$ (resp. $\rho'_a$ with $a\in T'$). Since the set of relations $\{\rho_y\mid y\in T\}$ is invariant under $G$ and its image in $\bk Q'$ under $\bk f$ is $\{\rho'_a\mid a\in T'\}$, $f$ induces a Galois covering $(Q,I_T)\rightarrow (Q',I'_{T'})$ with group $G$.

Based on the above discussion, we have the following result.

\begin{Prop}\label{prop-induced-covering}
Let $f:(Q,I)\rightarrow (Q',I')$ be a Galois covering of finite quivers with admissible relations with group $G$, and let $A$ and $A'$ be finite-dimensional algebras given by quivers with relations $(Q,I)$ and $(Q',I')$ respectively. Let $T'$ be a subset of $Q'_0$ such that $A'e_{i'}$ is injective and not uniserial for each $i'\in T'$, and let $T=f^{-1}(T')$. Then $Ae_{i}$ is injective and not uniserial for each $i\in T$. Moreover, let
$$S_{T}=\bigoplus\limits_{i\in T}\soc(Ae_i)$$
be an ideal of $A$ and
$$S'_{T'}=\bigoplus\limits_{i'\in T'}\soc(A'e_{i'})$$
be an ideal of $A'$. Then there exist an admissible ideal $I_T$ in $\bk Q$ and an admissible ideal $I'_{T'}$ in $\bk Q'$ such that $A/S_T\cong \bk Q/I_T$, $A'/S'_{T'}\cong \bk Q'/I'_{T'}$ and such that $f$ induces a Galois covering $(Q,I_T)\rightarrow (Q',I'_{T'})$ with group $G$.
\end{Prop}

As a consequence, we show that for two special biserial algebras $A$ and $B$, if there exists a Galois covering between them, then $A$ and $B$ (resp. $\gr(A)$ and $\gr(B)$) have the same representation type.

\begin{Cor}\label{cor-induced-covering}
	Let $A=\bk Q_A/I_A$ and $B=\bk Q_B/I_B$ be two special biserial algebras, $\overline{A}$ and $\overline{B}$ be the associated string algebras of $A$ and $B$ respectively. Let $\gr(A)$ and $\gr(B)$ be the associated graded algebras of
	$A$ and $B$ respectively, and let   $\overline{\gr(A)}$ and $\overline{\gr(B)}$ be the associated string algebras. If $f: (Q_A, I_A)\longrightarrow(Q_{B},I_{B})$ is a Galois covering  of quivers with relations with group $G$, then there exist ideals $I_{A,1}$, $I'_A$, $I_{A,2}$ in $\bk Q_A$ and ideals $I_{B,1}$, $I'_B$, $I_{B,2}$ in $\bk Q_B$ such that  $\overline{A}=\bk Q_A/I_{A,1}$, $\gr(A)=\bk Q_A/I_A'$, $\overline{\gr(A)}=\bk Q_A/I_{A,2}$ and $\overline{B}=\bk Q_B/I_{B,1}$, $\gr(B)=\bk Q_B/I_B'$, $\overline{\gr(B)}=\bk Q_B/I_{B,2}$, such that $f$ induces Galois coverings $(Q_A, I_{A,1})\longrightarrow(Q_{B},I_{B,1})$, $(Q_A, I'_A)\longrightarrow(Q_{B},I'_{B})$ and
	$(Q_A, I_{A,2})\longrightarrow(Q_{B},I_{B,2})$ with the same group. Therefore, 	
	\begin{enumerate}[$(1)$]
		\item  $\mathrm{Ba}(\overline{B})$ (resp. $\mathrm{Ba}(\overline{\gr(B)})$) is empty if and only if $\mathrm{Ba}(\overline{A})$ (resp. $\mathrm{Ba}(\overline{\gr(A)})$) is empty,
		\item The cardinality of $\mathrm{Ba}(\overline{B})$ (resp. $\mathrm{Ba}(\overline{\gr(B)})$) is a positive number if and only if the cardinality of $\mathrm{Ba}(\overline{A})$ (resp. $\mathrm{Ba}(\overline{\gr(A)})$) is a positive number.
	\end{enumerate}
In particular, $A$ and $B$ (resp. $\gr(A)$ and $\gr(B)$) have the same representation type.
\end{Cor}
\begin{proof}
By Lemma \ref{lem-grA-grM}, Propositions \ref{prop-covering-on-graded-categories}, \ref{prop-induced-covering}, there exist ideals $I_{A,1}$, $I'_A$, $I_{A,2}$ in $\bk Q_A$ and ideals $I_{B,1}$, $I'_B$, $I_{B,2}$ in $\bk Q_B$ such that  $\overline{A}=\bk Q_A/I_{A,1}$, $\gr(A)=\bk Q_A/I_A'$, $\overline{\gr(A)}=\bk Q_A/I_{A,2}$ and $\overline{B}=\bk Q_B/I_{B,1}$, $\gr(B)=\bk Q_B/I_B'$, $\overline{\gr(B)}=\bk Q_B/I_{B,2}$, such that $f$ induces Galois coverings $(Q_A, I_{A,1})\longrightarrow(Q_{B},I_{B,1})$, $(Q_A, I'_A)\longrightarrow(Q_{B},I'_{B})$ and
$(Q_A, I_{A,2})\longrightarrow(Q_{B},I_{B,2})$ with group $G$. In the following we only show $(1)$, $(2)$ for $\overline{\gr(A)}$ and $\overline{\gr(B)}$, since the proof is similar for $\overline{A}$ and $\overline{B}$.

    For the Galois covering $f: (Q_A, I_{A,2})\longrightarrow(Q_{B},I_{B,2})$, we add the definiton of $f$ on inverse arrows in $Q_A$ as follows:
	$$f(\alpha^{-1}):=f(\alpha)^{-1}.$$
	Moreover, since $f:Q_A\rightarrow Q_B$ is a Galois covering of quivers, we have a surjection map (which is also denoted by $f$) between the words in $\overline{\gr(A)}$ and the words in $\overline{\gr(B)}$. Since $f:(Q_A, I_{A,2})\longrightarrow(Q_{B},I_{B,2})$ is a Galois covering of quivers with relations, a path $p$ in $Q_A$ belongs to $I_{A,2}$ if and only if $f(p)$ belongs to $I_{B,2}$. Then $f$ restricts to a surjection map between the strings in $\overline{\gr(A)}$ and the strings in $\overline{\gr(B)}$.
	
	Since $f$ is a surjection map from the strings in
	$\overline{\gr(A)}$ to that of $\overline{\gr(B)}$ preserving the lengths of the corresponding strings,
	we have that all powers of $f(b)$ are strings for any band $b$ in $\overline{\gr(A)}$  and therefore there exist some string $b'$ in $\overline{\gr(B)}$ and a largest positive integer $m$ satisfying $f(b)=(b')^m$. Clearly, $b'$ is a band in $\overline{\gr(B)}$ and there is a map $g$ from the  bands in $\overline{\gr(A)}$ to the  bands in $\overline{\gr(B)}$ as follows.
	\begin{gather*}
		\{Bands~ in~ \overline{\gr(A)}\}	
		\xrightarrow{g}\{Bands~ in~ \overline{\gr(B)}\} \\
		b\longmapsto b',
	\end{gather*}
	\noindent where $b'$ is defined as above.
	
	It is straightforward to show that $g$ maps equivalent bands in $\overline{\gr(A)}$ to equivalent bands in $\overline{\gr(B)}$ and therefore $g$ induces the following map.
	\begin{gather*}
		\mathrm{Ba}(\overline{\gr(A)})	
		\xrightarrow{\overline{g}}\mathrm{Ba}(\overline{\gr(B)}) \\
		\qquad\overline{b}\longmapsto \overline{g(b)},
	\end{gather*}
	where $\overline{b}$ (resp. $\overline{g(b)}$) is the equivalence class of $b$ (resp. $g(b)$) in $\mathrm{Ba}(\overline{\gr(A)})$ (resp. $\mathrm{Ba}(\overline{\gr(B)})$).
	
	On the other hand, for any band $b'$ in
	$\overline{\gr(B)}$, since $b'$ is also a string and $f$ is surjective, there is a string $c_1$
	in $\overline{\gr(A)}$ satisfying $f(c_1)=b'$.
	There are two cases to be considered.
	
	If $s(c_1)=t(c_1)$, then $f((c_1)^m)={b'}^m$  for each positive integer $m$ and therefore all powers of  $c_1$
	are strings. Moreover, $c_1$ is not a power of a string of smaller length. Otherwise, it contradicts the condition that $b'$ is a band in $\overline{\gr(B)}$. Therefore, $c_1$ is a band in $\overline{\gr(A)}$ and $g(c_1)=b'$.
	
	If $s(c_1)\neq t(c_1)$, then there is a string $c_2$ with $s(c_2)=t(c_1)$ satisfying $f(c_2)=b'$. Note that $s(c_2)\neq t(c_2)$ and $c_2c_1$ is also a string in  $\overline{\gr(A)}$. Repeating the above steps for $c_2c_1$, since $Q_A$ is a finite quiver, it will end after finite steps and there is a band $b$ in $\overline{\gr(A)}$ satisfying $g(b)=b'$.
	
	Therefore, the above map $g$ is surjective and $g^{-1}(b')$ is a finite set for each band $b'$ in $\overline{\gr(B)}$. Thus  $\overline{g}:\mathrm{Ba}(\overline{\gr(A)})\rightarrow\mathrm{Ba}(\overline{\gr(B)})$ is surjective and  $\overline{g}^{-1}(\overline{b'})$ is a finite set for each band $b'$ in $\overline{\gr(B)}$. It follows that we have the desired result.
	
	Finally, the last statement follows from Theorem \ref{Theorem}.
\end{proof}

 It should be noted that although Corollary \ref{cor-induced-covering} is aimed at special biserial algebras, Propositions \ref{prop-covering-on-graded-categories} and \ref{prop-induced-covering}  hold for general finite-dimensional algebras.

\begin{Ex}
	Let $A=\bk Q/I$ with $Q$ the quiver
	$$\tikzset{every picture/.style={line width=0.75pt}}
	\begin{tikzpicture}[x=15pt,y=15pt,yscale=1,xscale=1]
		\node at(0,0) {$1$};
		\node at(5,0) {$2$};
		\node at(2.5,2.4) {$\alpha_1$};
		\node at(2.5,1.6) {$\alpha_2$};
		\node at(2.5,0.8) {$\beta_1$};
		\node at(2.5,0) {$\beta_2$};
		\node at(2.5,-0.8) {$\gamma_1$};
		\node at(2.5,-1.6) {$\gamma_2$};
		\draw[->] (0.2,2) -- (4.8,2);
		\draw[->] (4.8,1.2) -- (0.2,1.2);
		\draw[->] (0.2,0.4) -- (4.8,0.4);
		\draw[->] (4.8,-0.4) -- (0.2,-0.4);
		\draw[->] (0.2,-1.2) -- (4.8,-1.2);
		\draw[->] (4.8,-2) -- (0.2,-2);
	\end{tikzpicture}$$
	and
	\begin{align*}I=&\langle(\alpha_2\alpha_1)^2-\beta_2\beta_1,(\alpha_1\alpha_2)^2-\beta_1\beta_2,(\alpha_2\alpha_1)^2-\gamma_2\gamma_1,(\alpha_1\alpha_2)^2-\gamma_1\gamma_2,\\ &\beta_2\alpha_1,\beta_1\alpha_2,\alpha_2\beta_1,\alpha_1\beta_2,\gamma_2\alpha_1,\gamma_1\alpha_2,\alpha_2\gamma_1,\alpha_1\gamma_2,\gamma_2\beta_1,\gamma_1\beta_2,\beta_2\gamma_1,\beta_1\gamma_2\rangle,\end{align*}
	and let $A'=\bk Q'/I'$ with $Q'$ the quiver
	$$\tikzset{every picture/.style={line width=0.75pt}}
	\begin{tikzpicture}[x=15pt,y=15pt,yscale=1,xscale=1]
		\node at(0,0) {$a$};
		\node at(1.8,-1.2) {$\alpha$};
		\node at(-1.8,-1.2) {$\beta$};
		\node at(0,2) {$\gamma$};
		\draw[->]    (0.3,0) .. controls (3,0) and (1,-2) .. (0.1,-0.2) ;
		\draw[->]    (-0.1,-0.2) .. controls (-1,-2) and (-3,0) .. (-0.3,0) ;
		\draw[->]    (-0.1,0.2) .. controls (-1,2) and (1,2) .. (0.1,0.2) ;
	\end{tikzpicture}$$
	and
	$$I'=\langle\alpha^4-\beta^2,\alpha^4-\gamma^2,\beta\alpha,\alpha\beta,\gamma\alpha,\alpha\gamma,\gamma\beta,\beta\gamma\rangle.$$
	Note that these algebras are special multiserial but not special biserial. In fact, they are Brauer configuration algebras (see \cite{GS}), which form a subclass of {\bf $f_{s}$-BCA}s.
	
	Let $f:(Q,I)\rightarrow(Q',I')$ be the Galois covering with group $\mathbb{Z}/2\mathbb{Z}$ defined by $f(1)=f(2)=a$ and $f(\alpha_i)=\alpha$, $f(\beta_i)=\beta$, $f(\gamma_i)=\gamma$ for $i=1,2$. The ideal $L$ in $\bk Q$ and the ideal $L'$ in $\bk Q'$ in Proposition \ref{prop-covering-on-graded-categories} such that $\gr(A)\cong\bk Q/L$, $\gr(A')\cong\bk Q'/L'$ are given by
	\begin{align*}L=&\langle\alpha_1(\alpha_2\alpha_1)^2,\alpha_2(\alpha_1\alpha_2)^2,\beta_2\beta_1,\beta_1\beta_2,\gamma_2\gamma_1,\gamma_1\gamma_2,\beta_2\alpha_1,\beta_1\alpha_2,\\ &\alpha_2\beta_1,\alpha_1\beta_2,\gamma_2\alpha_1,\gamma_1\alpha_2,\alpha_2\gamma_1,\alpha_1\gamma_2,\gamma_2\beta_1,\gamma_1\beta_2,\beta_2\gamma_1,\beta_1\gamma_2\rangle,\end{align*}
	$$L'=\langle\alpha^5,\beta^2,\gamma^2,\beta\alpha,\alpha\beta,\gamma\alpha,\alpha\gamma,\gamma\beta,\beta\gamma\rangle.$$
	It is straightforward to show that $f$ induces a Galois covering $(Q,L)\rightarrow(Q',L')$ with group $\mathbb{Z}/2\mathbb{Z}$.
	
	Let $T'=Q'_0$ and $T=Q_0$. Then the ideal $S_T$ of $A$ (resp. $S'_{T'}$ of $A'$) in Proposition \ref{prop-induced-covering} is $\soc(A)$ (resp. $\soc(A')$), and $A/S_T\cong \bk Q/I_T$, $A'/S'_{T'}\cong \bk Q'/I'_{T'}$, where
		\begin{align*}I_T=&\langle(\alpha_2\alpha_1)^2,\beta_2\beta_1,\gamma_2\gamma_1,(\alpha_1\alpha_2)^2,\beta_1\beta_2,\gamma_1\gamma_2,\\ &\beta_2\alpha_1,\beta_1\alpha_2,\alpha_2\beta_1,\alpha_1\beta_2,\gamma_2\alpha_1,\gamma_1\alpha_2,\alpha_2\gamma_1,\alpha_1\gamma_2,\gamma_2\beta_1,\gamma_1\beta_2,\beta_2\gamma_1,\beta_1\gamma_2\rangle,\end{align*}
		and
		$$I'_{T'}=\langle\alpha^4,\beta^2,\gamma^2,\beta\alpha,\alpha\beta,\gamma\alpha,\alpha\gamma,\gamma\beta,\beta\gamma\rangle.$$
		It is straightforward to show that $f$ induces a Galois covering $(Q,I_T)\rightarrow(Q',I'_{T'})$ with group $\mathbb{Z}/2\mathbb{Z}$.
\end{Ex}

\section{Biserial fractional Brauer graph algebras and their representation types}\label{bfbgas}

\subsection{Biserial fractional Brauer graph algebras}\label{sec-bfbga}
\

In this subsection we recall the definition of biserial fractional Brauer graph algebras (see \cite{LL1, LL3}). 

Let $\Ga$ be a graph with vertex set $V(\Ga)$ and edge set $E(\Ga)$. We denote by ${\bf val}(v)$ the valency or the ordinary degree of a vertex $v\in V(\Ga)$; it is defined to be the number of edges in $\Ga$ incident to $v$, with the convention that a loop is counted twice in ${\bf val}(v)$.

\begin{Def} \label{Def-f-BG}
	A biserial fractional Brauer graph ($abbr$. ${\bf bFBG}$) $\Ga=(V(\Ga),E(\Ga),\bm,\mathfrak{o})$ is a finite connected (undirected)
	 graph with vertex set $V(\Ga)$ and edge set $E(\Ga)$ together with a multiplicity function $\bm:V(\Ga)\rightarrow \mathbb{Q}_{>0}$ and an orientation $\mathfrak{o}$ of $\Ga$ (that is, for each vertex $v$, there is a cyclic ordering of the edges incident to $v$), such that
\begin{enumerate}[$(1)$]
\item for each vertex $v$ of $\Gamma$, $\bm(v){\bf val}(v)$ is a positive integer;
\item for each edge $v\defeqi w$, the edge rotated around $v$ from $i$ after $\bm(v){\bf val}(v)$ steps is equal to the edge rotated around $w$ from $i$ after $\bm(w){\bf val}(w)$ steps.
\end{enumerate}
\end{Def}

Note that ${\bf bFBG}$ is named as fractional Brauer graph algebras of type $MS$ (abbr. {\bf $f_{ms}$-BG}) in \cite{LL1} (see also Remark \ref{change-name}).

\begin{Rem}
    \begin{enumerate}[$(1)$]
        \item We will always display a ${\bf bFBG}$ $\Ga$ as a graph with the edges incident to each vertex appearing in anticlockwise direction to reflect the ordering and  often leave out the symbol $\mathfrak{o}$.
        \item An edge $i\in E(\Ga)$ is said to be truncated at a vertex $v$ if $i$ is incident to $v$ such that $\bm(v)\val(v)=1$.
        \item Since an edge in a graph can be seen as a union of two half-edges, the orientation $\mathfrak{o}$ of $\Ga$ can be considered as cyclic orderings of half-edges incident to each vertex.
        \item The Nakayama automorphism of a ${\bf bFBG}$ $\Ga$, denote by $\nu$, is defined as the graph automorphism of $\Ga$ (see Section \ref{sec-rep-type-bFBGA} for explanation) such that each half-edge $h$ incident to vertex $v$ is mapped to the half-edge $h'$ rotated around $v$ from $h$ after $\bm(v){\bf val}(v)$ steps.
    \end{enumerate}
\end{Rem}

We give an example to illustrate the above notions.

\begin{Ex} \label{Example-equi-def}
	Let $\Gamma$ be a ${\bf bFBG}$ given by the following graph
	$$\xymatrix{
		a \ar@(ul,ur)@{-}[]^{1}
		\ar@(u,u)@{-}[r]^{3} \ar@{-}[r]^{2}
		& b
	}$$
	\noindent with $\bm(a)=\bm(b)=\frac{1}{2}$. Note that the cyclic ordering of the edges incident to vertex $a$ is given by $1<2<1<3<1$, to vertex $b$ is given by $2<3<2$. Moreover, the valency of vertex $a$ (resp. $b$) is $4$ (resp. $2$). Write each edge $i=(i',i'')$ where $i'$ and $i''$ denotes the half-edges from left to right, then the Nakayama automorphism $\nu$ of $\Gamma$ is given by: $1'\mapsto 1''\mapsto 1', 2'\mapsto 3'\mapsto 2', 2''\mapsto 3''\mapsto 2''$.
\end{Ex}

Recall that a {\it Brauer graph} (abbr. ${\bf BG}$) is a finite connected (undirected)
graph $\Ga=(V(\Ga),E(\Ga),\bm,\mathfrak{o})$, where $V(\Ga)$ is the vertex set, $E(\Ga)$ is the edge set, $\bm:V(\Ga)\rightarrow\mathbb{Z}_{>0}$ is the multiplicity function, and $\mathfrak{o}$ is the orientation of $\Ga$ (that is, for each vertex $v$, there is a cyclic ordering of the edges incident to $v$). We often leave out the symbol $\mathfrak{o}$. A ${\bf BG}$ is called {\it multiplicity-free} if the multiplicity of each vertex in the ${\bf BG}$ is $1$. Note that a ${\bf BG}$ is exactly a ${\bf bFBG}$ with integral multiplicity function.

A {\it Brauer tree} (abbr. ${\bf BT}$) is a ${\bf BG}$ $\Ga=(V(\Ga),E(\Ga),\bm)$ such that $(V(\Ga),E(\Ga))$ is a tree and $\bm(v)=1$ for all but at most one $v\in V(\Ga)$. In this case we always choose a specified vertex $v_{0}$ (if $\bm(v)>1$ then we choose $v_0=v$), called the exceptional vertex, whose multiplicity will be denoted by $\bm_{0}$.

Recall that any ${\bf bFBG}$ $\Gamma$ determines a finite-dimensional basic self-injective $\bk$-algebra $A_{\Gamma}=\bk Q_{\Gamma}/I_{\Gamma}$ called biserial fractional Brauer graph algebra (abbr. ${\bf bFBGA}$). In particular, $A_{\Gamma}$ is also called a Brauer graph algebra (abbr. ${\bf BGA}$) if $\Gamma$ is a ${\bf BG}$. Precisely, let $\Ga=(V(\Ga),E(\Ga),d)$ be a ${\bf bFBG}$, we will construct a quiver $Q_{\Gamma}$ and an admissible ideal $I_{\Gamma}$ in $\bk Q_{\Gamma}$ as follows.

The set $(Q_{\Gamma})_0$ of vertices is given by the set of edges $E(\Ga)$ of $\Ga$, denoting the vertex in $(Q_{\Gamma})_0$ corresponding to the edge $i$ in $E(\Ga)$ also by $i$. The set $(Q_{\Gamma})_1$ of arrows is given by the cyclic ordering in $\Ga$. If $\Ga$ is the ${\bf bFBG}$ given by $n$ edges with two vertices such that each edge is truncated at both of them (or if $\Ga$ is the ${\bf bFBG}$ given by $n$ loops with a single vertex such that each loop is truncated at it), suppose that the cyclic ordering at a vertex (or at the unique vertex, respectively) of $\Ga$ is given by $i_1<i_2<\cdots<i_n<i_1$ (or by $i_1<i_2<\cdots<i_n<i_1<i_2<\cdots<i_n<i_1$, respectively), then there is a unique arrow $i_k\rightarrow i_{k+1}$ in $(Q_{\Gamma})_1$ for $1\leq k\leq n$ and the ideal $I_{\Gamma}$ is generated by paths of length $2$, where $i_{n+1}=i_1$. The corresponding ${\bf bFBGA}$ is isomorphic to self-injective Nakayama algebra of Loewy length $2$. In the following discussion, we exclude the above cases. Suppose that the cyclic ordering at a vertex $v$ of $\Ga$ is given by $i_1<i_2<\cdots<i_n<i_1$; note that we might have $i_j=i_k$ for some $j\neq k$, if some of the edges are loops. We call
$i_{j+1}$  the successor of $i_j$ for $1\leq j\leq n-1$ and $i_1$ the successor of $i_n$. Note that if $v$ is a vertex incident to edge $i$ with $\val(v)=1$ and if $\bm(v)>1$ then
the successor of $i$ is $i$, if $\bm(v)=1$ then $i$ does not have a successor. If $i$ and $j$ are two edges in $E(\Ga)$ incident to a common vertex $v$ and such that $j$ is a successor of $i$ in the cyclic ordering of the edges at $v$ with $\bm(v)\val(v)>1$, then there is an arrow $i\rightarrow j$ in $(Q_{\Gamma})_1$.

Since every arrow of $Q_{\Gamma}$ starts and ends at an edge of $\Ga$, there are at most two arrows starting and ending at every vertex of $Q_{\Gamma}$. Every vertex $v\in V(\Ga)$ with $\bm(v)\val(v)>1$ and every edge $v\defeqi w$ incident to $v$, gives rise to a path $P_v$ of length $\bm(v)\val(v)$ starting at $i$  in $Q_{\Gamma}$.

We define the ideal $I_{\Gamma}$ in the path algebra $\bk Q_{\Gamma}$ generated by three types of relations. In the following we identify the set of edges $E(\Ga)$ of $\Ga$ with the set of vertices $(Q_{\Gamma})_0$ of the corresponding quiver $Q_{\Gamma}$.

Relation of the first type:
$$P_v-P_w,$$
for any edge  $v\defeqi w$ in $E(\Ga)$ with $\bm(v)\val(v)>1$ and $\bm(w)\val(w)>1$ (since the edge rotated around $v$ from $i$ after $\bm(v){\bf val}(v)$ steps is equal to the edge rotated around $w$ from $i$ after $\bm(w){\bf val}(w)$ steps, $P_v-P_w$ is a uniform relation).

Relation of the second type:
$$\alpha P_v$$
where  $t(\alpha)$ is a successor of $t(P_v)$ in the cyclic ordering of the edges at $v$ for any edge  $v\defeqi w$ in $E(\Ga)$ with $\bm(v)\val(v)>1$.

Relation of the third type:
$$\beta\alpha,$$
for any $\alpha,\beta\in (Q_{\Gamma})_1$ such that there does not exist a vertex $v$ in $\Ga$ satisfying  $\beta\alpha$ is a subpath of $P_v$.

We note that $I_{\Gamma}$ is an admissible ideal and the relations generating $I_{\Gamma}$ do not constitute a minimal set of relations.
Many of the relations, in particular many of the relations of the second type, are redundant.

\begin{Rem} \label{change-name}
In \cite{LL1}, ${\bf bFBGA}$ is named as fractional Brauer graph algebra of type $MS$ (abbr. {\bf $f_{ms}$-BGA}), since the more general class of fractional Brauer configuration algebra of type $MS$ (abbr. {\bf $f_{ms}$-BCA}) are special multiserial and self-injective. As the subclass consisting of {\bf $f_{ms}$-BGA}s are self-injective special biserial and proper generalization of ${\bf BGA}$s, we choose in the present paper the name ${\bf bFBGA}$ as suggested by an anonymous reviewer.
\end{Rem}

The following theorem collects some general properties on ${\bf bFBG}$s and ${\bf bFBGA}$s.

\begin{Thm}	[cf.{ \cite[Section 4 and Section 5]{LL1}}]
	\label{general-properties-of-BGA}
	\begin{enumerate}[$(1)$]
		\item Given a ${\bf bFBG}$ $\Gamma$, the associated ${\bf bFBGA}$ $A_{\Gamma}$ is finite-dimensional, basic, self-injective (where the Nakayama automorphism $\nu$ of $\Ga$ induces a Nakayama automorphism of the self-injective algebra $A_{\Ga}$), and indecomposable.
		\item  ${\bf bFBGA}$s are special biserial and of tame representation type.
		\item ${\bf bFBGA}$s are symmetric algebras if and only if they are ${\bf BGA}$s.
	\end{enumerate}
\end{Thm}

Since ${\bf bFBGA}$s are special biserial, they are biserial and the composition factors of the maximal uniserial submodules of the
indecomposable projective modules can be read from the defining ${\bf bFBG}$s. Let $\Ga$ be a ${\bf bFBG}$ and $A$ the associated ${\bf bFBGA}$. There is a one-to-one correspondence between the edges $i$ in $\Ga$ and the simple $A$-modules $S_i$ such that the projective cover $P_i$ of $S_i$ has the following structure. We have that $\rad(P_i)/\soc(P_i)$ is a direct sum of two (possibly zero) uniserial modules
$V_v$ and $V_w$ corresponding to the two vertices $v$ and $w$ incident to the edge $i$.  Suppose that $i$ is not truncated. Let $i,i_1,\cdots,i_{\val(v)-1}$ be the successor sequence for $i$ at $v$, and $i,j_1,\cdots,j_{\val(w)-1}$ the successor sequence for $i$ at $w$, using division with remainder, $\bm(v)\val(v)=q(v)\val(v)+r(v)$ and $\bm(w)\val(w)=q'(w)\val(w)+r'(w)$. Then $\soc(P_i)\cong S_{i_{r(v)}}=S_{j_{r'(w)}}$ where $i_0=j_0=i$, $V_v$ and $V_w$ have composition series
$$S_{i_1},\cdots,S_{i_{\val(v)-1}},S_i,S_{i_1},\cdots,S_{i_{\val(v)-1}},\cdots,S_i,S_{i_1},\cdots,S_{i_{r(v)-1}}$$
and
$$S_{j_1},\cdots,S_{j_{\val(w)-1}},S_i,S_{j_1},\cdots,S_{j_{\val(w)-1}},\cdots,S_i,S_{j_1},\cdots,S_{j_{r'(w)-1}}$$
respectively, such that, for $k=r(v)+1,\cdots,\val(v)-1$, the simple module $S_{i_k}$ occurs precisely $q(v)$ times, and, for $l=r'(w)+1,\cdots,\val(w)-1$, the simple module $S_{j_l}$ occurs precisely $q'(w)$ times.
In the case where $i$ is truncated, then $P_i$ is itself uniserial. Suppose that $i$ is not truncated at $v$ and is truncated at $w$. Then $P_i$ has composition series
$$S_i,S_{i_1},\cdots,S_{i_{\val(v)-1}},S_i,S_{i_1},\cdots,S_{i_{\val(v)-1}},\cdots,S_i,S_{i_1},\cdots,S_{i_{r(v)}}$$
 and for $k=r(v)+1,\cdots,\val(v)-1$, the simple module $S_{i_k}$ occurs precisely $q(v)$ times.

\begin{Ex}\label{Example-BT}
$(1)$ (see{ \cite[Example 2.37]{LL3}})	Let $\Ga$	be the following  ${\bf bFBG}$
\begin{center}
\begin{tikzpicture}
\draw[-] (0.45,0.125) arc (15:345:0.5);
\draw[-] (0.625,0.05) arc (-75:255:0.5);
\node at(0.5,0) {$a$};
\node at(-0.65,0) {$1$};
\node at(0.5,1.2) {$2$};
\end{tikzpicture}
\end{center}
with $\bm(a)=\frac{1}{2}.$  The cyclic ordering of the edges incident to vertex $a$ is given by $1<2<1<2        <1$. We have $\val(a)=4$.
The quiver $Q$ of the corresponding ${\bf bFBGA}$:

$$\xymatrix@R-1em{
1\ar@{->}@<2pt>[r]^{\alpha_1}\ar@/_1.7pc/@<-1pt>[r]_{\alpha_3}&2 \ar@{->}@<2pt>[l]^{\alpha_2}\ar@/_1.7pc/@<-1pt>[l]_{\alpha_4}
}
$$

Set of relations of the three types of the ${\bf bFBGA}$:

The first type: $\alpha_2\alpha_1-\alpha_4\alpha_3,\alpha_3\alpha_2-\alpha_1\alpha_4.$

The second type: $\alpha^3$ (By abuse of notation, we write $\alpha^3$ for $\alpha_3\alpha_2\alpha_1$, etc.)

The third type: $\alpha_4\alpha_1,\alpha_1\alpha_2,\alpha_2\alpha_3,\alpha_3\alpha_4.$

The regular representation of the ${\bf bFBGA}$ is as follows:

$$\begin{picture}(30.00,45.00)
	\unitlength=1.00mm \special{em:linewidth 0.4pt}
	\linethickness{0.4pt}
	\put(-21,15){$1$} \put(-17,8){$2$}
	\put(-25,8){$2$}\put(-21,0){$1$}
	\put(-13,8){$\oplus$}
	\put(-5,15){$2$} \put(-1,8){$1$}
	\put(-9,8){$1$}\put(-5,0){$2$}
\end{picture}$$

Note that the above ${\bf bFBGA}$ is a weakly symmetric algebra, not symmetric (if $A=A_{\Ga}$ is symmetric with a symmetric linear form $\epsilon$, then $\epsilon(\alpha_2\alpha_1)\neq 0$ since $\mathrm{soc}(Ae_1)=\bk\alpha_2\alpha_1$ is a nonzero left ideal of $A$, but $\epsilon(\alpha_1\alpha_2)=0$ since $\alpha_1\alpha_2=0$ in $A$, which contradicts our assumption that $\epsilon$ is symmetric).
\\
$(2)$ Let $\Ga$ be the  following ${\bf bFBG}$
\begin{center}
\begin{tikzpicture}[x=15pt,y=15pt,yscale=1,xscale=1]
\draw[-] (0.25,0.1) arc (-75:255:1);
\node at(0,0) {$a$};
\node at(4,0) {$b$};
\node at(6,3.464) {$c$};
\node at(6,-3.464) {$e$};
\node at(0,2.4) {$1$};
\node at(2,1.9) {$2$};
\node at(2,-1.9) {$3$};
\node at(3.4,2.75) {$4$};
\node at(3.4,-2.75) {$5$};
\node at(6.6,0.72) {$6$};
\node at(6.6,-0.72) {$7$};
\draw    (0,0.3) .. controls (0,2) and (4,2) .. (4,0.3) ;
\draw    (0,-0.3) .. controls (0,-2) and (4,-2) .. (4,-0.3) ;
\draw    (3.74,0.15) .. controls (2.268,1) and (4.268,4.464) .. (5.74,3.614) ;
\draw    (3.74,-0.15) .. controls (2.268,-1) and (4.268,-4.464) .. (5.73,-3.614) ;
\draw    (4.26,-0.15) .. controls (5.732,-1) and (7.732,2.464) .. (6.26,3.314) ;
\draw    (4.26,0.15) .. controls (5.732,1) and (7.732,-2.464) .. (6.26,-3.314) ;
\end{tikzpicture}
\end{center}
with $\bm(a)=\bm(b)=\bm(c)=\bm(e)=\frac{1}{2}$. The cyclic ordering of the edges incident to vertex $a$ is given by $1<2<1<3<1$, to vertex $b$ is given by $2<4<5<3<6<7<2$, to vertex $c$ is given by $4<6<4$, to vertex $e$ is given by $5<7<5$. We have $\val(a)=4$, $\val(b)=6$, $\val(c)=\val(e)=2$. Note that the edges $4,6$ (resp. $5,7$) are truncated at the vertex $c$ (resp. $e$).

The quiver $Q$ of the corresponding ${\bf bFBGA}$:

$$\xymatrix@R-1em{
	5\ar@{->}[d]_{\beta_3}& & 4\ar@{->}[ll]_{\beta_2}\\
	3\ar@{->}@<-2pt>[r]_{\alpha_4}\ar@{->}[d]_{\beta_4}	&1\ar@{->}@<2pt>[r]^{\alpha_1}\ar@{->}@<-2pt>[l]_{\alpha_3}&2 \ar@{->}@<2pt>[l]^{\alpha_2}\ar@{->}[u]_{\beta_1}\\
	6\ar@{->}[rr]_{\beta_5} & & 7\ar@{->}[u]_{\beta_6}
}
$$

Set of relations of the three types of the ${\bf bFBGA}$:

The first type: $\alpha_2\alpha_1-\alpha_4\alpha_3,\alpha_3\alpha_2-\beta_3\beta_2\beta_1,\alpha_1\alpha_4-\beta_6\beta_5\beta_4.$

The second type: $\alpha^3,\beta^4.$ (By abuse of notation, we write $\alpha^3$ for $\alpha_3\alpha_2\alpha_1$, etc.)

The third type: $\beta_1\alpha_1,\alpha_1\alpha_2,\beta_4\alpha_3,\alpha_3\alpha_4,\alpha_4\beta_3, \alpha_2\beta_6.$

The regular representation of the ${\bf bFBGA}$ is as follows:

$$\begin{picture}(30.00,45.00)
	\unitlength=1.00mm \special{em:linewidth 0.4pt}
	\linethickness{0.4pt}
		\put(-21,15){$1$} \put(-17,8){$3$}
	\put(-25,8){$2$}\put(-21,0){$1$}
	\put(-13,8){$\oplus$}
	\put(-5,15){$2$} \put(-1,10){$4$}
\put(-9,8){$1$}\put(-1,5){$5$}\put(-5,0){$3$}
\put(3,8){$\oplus$} \put(11,15){$3$} \put(15,10){$6$} \put(15,5){$7$}\put(11,0){$2$}
\put(7,8){$1$}
\put(19,8){$\oplus$} \put(23,15){$4$} \put(23,10){$5$} \put(23,5){$3$} \put(23,0){$6$}
\put(27,8){$\oplus$} \put(31,15){$5$} \put(31,10){$3$} \put(31,5){$6$} \put(31,0){$7$}
\put(35,8){$\oplus$}
\put(39,15){$6$} \put(39,10){$7$} \put(39,5){$2$} \put(39,0){$4$}
\put(43,8){$\oplus$}
\put(47,15){$7$} \put(47,10){$2$} \put(47,5){$4$} \put(47,0){$5$}
\end{picture}$$
\\
$(3)$ Let $\Ga$ be the following ${\bf bFBG}$

$$\xymatrix{
	a
	\ar@(u,u)@{-}[r]^{1}	\ar@{-}[r]^{2}
	& b \ar@{-}[r]^{3}	\ar@(ul,u)@{-}[r]^{4} &c}
$$

\noindent with $\bm(a)=\frac{3}{2}, \bm(b)=\bm(c)=\frac{1}{2}$. The cyclic ordering of the edges incident to vertex $a$ is given by $1<2<1$, to vertex $b$ is given by $1<4<2<3<1$, to vertex $c$ is given by $3<4<3$. We have $\val(a)=\val(c)=2$, $\val(b)=4$. Note that the edges $3,4$ are truncated at the vertex $c$.

The quiver $Q$ of the corresponding ${\bf bFBGA}$:

$$\xymatrix@R-1em{
4\ar@{->}[dr]_{\beta_2} &	1\ar@{->}@<2pt>[d]^{\alpha_1}\ar@{->}[l]_{\beta_1} &3 \ar@{->}[l]_{\beta_4} \\
& 	2\ar@{->}@<2pt>[u]^{\alpha_2}\ar@{->}[ur]_{\beta_3}  &
}
$$

Set of relations of the three types of the ${\bf bFBGA}$:

The first type: $\alpha_1\alpha_2\alpha_1-\beta_2\beta_1, \alpha_2\alpha_1\alpha_2-\beta_4\beta_3.$

The second type: $\alpha^4,\beta^3.$ (By abuse of notation, we write $\alpha^4$  for $\alpha_2\alpha_1\alpha_2\alpha_1$, etc.)

The third type: $\beta_3\alpha_1,\beta_1\alpha_2, \alpha_2\beta_2,\alpha_1\beta_4.$

The regular representation of the ${\bf bFBGA}$ is as follows:

$$\begin{picture}(30.00,40.00)
	\unitlength=1.00mm \special{em:linewidth 0.4pt}
	\linethickness{0.4pt}
	\put(-10,15){$1$} \put(-6,8){$4$}
	\put(-14,10){$2$}\put(-14,5){$1$}\put(-10,0){$2$}
	\put(-2,8){$\oplus$} \put(6,15){$2$} \put(2,10){$1$} \put(2,5){$2$}\put(6,0){$1$}
	\put(10,8){$3$}
	\put(14,8){$\oplus$} \put(18,15){$3$} \put(18,8){$1$} \put(18,0){$4$}
	\put(22,8){$\oplus$} \put(26,15){$4$} \put(26,8){$2$} \put(26,0){$3$}
\end{picture}$$
\end{Ex}

\subsection{The representation type of biserial fractional Brauer graph algebras}\label{sec-rep-type-bFBGA}
\

In this subsection, we recall some notions from \cite{LL3} and reinterpret some results on ${\bf bFBG}$s and ${\bf bFBGA}$s. For this purpose, it is necessary to use the notion of half-edges in a graph.

For two graphs $\Ga=(V(\Ga),E(\Ga))$ and $\Ga'=(V(\Ga'),E(\Ga'))$, denote by $H(\Ga)$ and $H(\Ga')$ the set of half-edges of $\Ga$ and $\Ga'$ respectively, denote by $v:H(\Ga)\rightarrow V(\Gamma)$ (resp. $v':H(\Ga')\rightarrow V(\Gamma')$) the map sending each half-edge in $\Ga$ (resp. $\Ga'$) to the vertex that it incidents to, and denote by $\iota:H(\Ga)\rightarrow H(\Ga)$ (resp. $\iota':H(\Ga')\rightarrow H(\Ga')$) the involution mapping each half-edge to the other half-edge belonging to the same edge. Recall that a {\it morphism of graphs} $f:\Ga\rightarrow\Ga'$ is a map $f:H(\Ga)\rightarrow H(\Ga')$ such that for each two half-edges $h_1$, $h_2$ in $H(\Ga)$, $v'(f(h_1))=v'(f(h_2))$ if $v(h_1)=v(h_2)$; and for each half-edge $h$ in $\Ga$, $\iota'(f(h))=f(\iota(h))$. Note that $f$ induces two maps $f_0:V(\Ga)\rightarrow V(\Ga')$ and $f_1:E(\Ga)\rightarrow E(\Ga')$ between the sets of vertices and edges of $\Gamma$ and $\Gamma'$ respectively, such that if $V$ is the set of vertices of $\Gamma$ that incident to an edge $e\in E(\Ga)$, then $f_0(V)$ is the set of vertices of $\Gamma'$ that incident to the edge $f_1(e)\in E(\Ga')$. Therefore the above definition of morphism of graphs using half-edges is equivalent to the usual definition of morphism of graphs.

An {\it automorphism} $f$ of a graph $\Gamma$ is a morphism $f:\Ga\rightarrow\Ga$ such that $f$ is a permutation of $H(\Gamma)$. If $\Ga$ is a ${\bf bFBG}$, then the Nakayama automorphism of $\Ga$ is indeed an automorphism of $\Ga$ as a graph.

\begin{Rem}
	The definition of morphism of graphs given as above form is inspired by morphism of fractional Brauer configurations defined in \cite{LL2}.
\end{Rem}

\begin{Def}[see {\cite[Definition 2.3]{AGH}}]\label{def-half-edged-graph}
	A half-edged graph $\Ga=(V(\Ga),E(\Ga),O(\Ga))$ is a finite connected graph $\Ga=(V(\Ga),E(\Ga))$ (where $V(\Ga)$ and $E(\Ga)$ are the set of vertices and the set of edges respectively) together with a finite set $O(\Gamma)$, where each $h\in O(\Gamma)$ is a half-edge (called isolated half-edge) which does not belong to any edge in $E(\Ga)$ and which incidents to some vertex in $V(\Gamma)$. We call a half-edge of $\Ga$ non-isolated if it belongs to some edge in $E(\Ga)$.
\end{Def}

If $\Ga=(V(\Ga),E(\Ga))$ is a finite connected graph and $G$ is a group of automorphisms of $\Gamma$, then the quotient $\Gamma/G$ is defined as follows and it is a half-edged graph.

Denote by $G_0=\{f_0\mid f\in G\}$ the associated permutation group on the set $V(\Gamma)$. The vertex set $V(\Gamma/G)$ of $\Gamma/G$ is the set $V(\Ga)/G_0$ of $G_0$-orbits of $V(\Gamma)$. Denote by $\iota$ the involution of $H(\Gamma)$ mapping each half-edge $h$ to the unique different half-edge $\iota(h)$ belonging to the same edge. Denote by $H(\Gamma)/G$ the set of $G$-orbits of half-edges, and let $\overline{\iota}$ be the involution of $H(\Gamma)/G$ induced by $\iota$ (note that $\overline{\iota}$ may have fixed points). The edges of $\Gamma/G$ are given by $\overline{\iota}$-orbits of cardinality $2$, and the isolated half-edges of $\Gamma/G$ are given by $\overline{\iota}$-fixed points, where an edge $\{G\cdot h,\overline{\iota}(G\cdot h)\}$ is incident to vertices $G_0\cdot v(h)$ and $G_0\cdot v(\iota(h))$, and an isolated half-edge $G\cdot h$ is incident to the vertex $G_0\cdot v(h)$. We should mention that this quotient is consistent with the quotient of a Brauer $G$-set in \cite{LL3} and when we apply to a ${\bf bFBG}$ $\Gamma$ with $G$ the group of automorphisms of $\Ga$ generated by the Nakayama automorphism $\nu$ of $\Ga$, then the quotient $\Gamma/G$ is indeed a modified Brauer graph, see Definition \ref{def-bg-with-orbifolds} and Lemma \ref{lem-quotient-being-mBG} below.

\begin{Rem}\label{remark-half-edge}
	We will consider a half-edge in a graph $\Ga=(V(\Ga),E(\Ga))$ as a pair $(e,v)$, where $e\in E(\Gamma)$ and $v\in V(\Gamma)$ such that $e$ is incident to $v$ (if $e$ is a loop incident to a vertex $v$, we denote the two half-edges associated to $e$ by $(e,v)$ and $(e,v)'$).
\end{Rem}

\begin{Ex}\label{ex-quotient-of-a-graph}
	Let $\Gamma=(V(\Gamma),E(\Gamma))$ be the graph
	$$\begin{tikzpicture}
		\fill (-1,0) circle (0.5ex);
		\fill (1,0) circle (0.5ex);
		\fill (0,2) circle (0.5ex);
		\draw (-1,0) -- (1,0);
		\draw (-1,0) -- (0,2);
		\draw (1,0) -- (0,2);
		\node at(-1.25,0) {$v_1$};
		\node at(0,2.2) {$v_2$};
		\node at(1.3,0) {$v_3$};
		\node at(-0.7,1) {$e_1$};
		\node at(0.7,1) {$e_2$};
		\node at(0,-0.2) {$e_3$};
	\end{tikzpicture}$$
	and let $f:H(\Gamma)\rightarrow H(\Gamma)$ and $f':H(\Gamma)\rightarrow H(\Gamma)$ be the automorphism of $\Gamma$ given by
	$$f(e_1,v_1)=(e_1,v_2),~f(e_1,v_2)=(e_1,v_1),~f(e_2,v_2)=(e_3,v_1),$$
	$$f(e_2,v_3)=(e_3,v_3),~f(e_3,v_1)=(e_2,v_2),~f(e_3,v_3)=(e_2,v_3)$$
	and
	$$f'(e_1,v_1)=(e_2,v_2),~f'(e_1,v_2)=(e_2,v_3),~f'(e_2,v_2)=(e_3,v_3),$$
	$$f'(e_2,v_3)=(e_3,v_1),~f'(e_3,v_3)=(e_1,v_1),~f'(e_3,v_1)=(e_1,v_2)$$
	respectively. Denote by $G$ (resp. $G'$) the group of automorphisms of $\Gamma$ generated by $f$ (resp. $f'$). Then $\Gamma/G$ is a half-edged graph given by the diagram
	$$\begin{tikzpicture}
		\fill (0,0) circle (0.5ex);
		\fill (1,0) circle (0.5ex);
		\draw (-0.5,0) -- (1,0);
	\end{tikzpicture}$$
	and $\Gamma/G'$ is the graph
	$$\begin{tikzpicture}
		\draw (0.5,0) circle (0.5);
		\fill (0,0) circle (0.5ex);
	\end{tikzpicture}.$$
\end{Ex}

The modified Brauer graph is introduced in \cite{LL3} using the language of Brauer $G$-set. Since we do not need the
general notion of Brauer $G$-set, in the present paper we define modified Brauer graph directly using half-edged graph as follows. Here we emphasize more on the subset of isolated half-edges.

\begin{Def} [see {\cite[Definition 2.7]{LL3}}]\label{def-bg-with-orbifolds}
	A modified Brauer graph (abbr. ${\bf mBG}$) $(V(\Ga),E(\Ga),O(\Ga),\bm,\mathfrak{o})$ is a half-edged graph $\Ga=(V(\Ga),E(\Ga),O(\Ga))$, where $V(\Ga)$ is the set of vertices, $E(\Ga)$ is the set of edges, $O(\Ga)$ is the set of isolated half-edges, together with a multiplicity function $\bm:V(\Ga)\rightarrow\mathbb{Z}_{>0}$ and an orientation $\mathfrak{o}$ of $\Ga$, that is, for each vertex $v$, a cyclic ordering of the half-edges (both isolated and non-isolated) incident to $v$.
\end{Def}

Note that an ${\bf mBG}$ is a special kind of Brauer $G$-set in \cite{LL3} and a ${\bf BG}$ $\Ga$ is exactly an ${\bf mBG}$ with empty $O(\Ga)$. Similar to ${\bf BG}$, we call an ${\bf mBG}$ $\Ga$ {\it multiplicity-free} if $\bm\equiv 1$.

\begin{Lem}[see {\cite[Lemma 2.32]{LL3}}]\label{lem-quotient-being-mBG}
Let $\Ga=(V(\Ga),E(\Ga),\bm,\mathfrak{o})$ be a ${\bf bFBG}$ and let $\ang{\nu}$ the group of automorphisms of $\Ga$ generated by the Nakayama automorphism $\nu$ of $\Ga$. Then the quotient $\Gamma/\ang{\nu}$ becomes an ${\bf mBG}$ in a natural way.
\end{Lem}

\begin{proof}
For the convenience of the readers, we include a proof here. Since $\Ga$ is connected, the group $G:=\ang{\nu}$ acts freely on $\Gamma$ (cf. \cite[Lemma 3.6]{LL3}), that is, for each $f\in G$ and each $h\in H(\Gamma)$, $f(h)=h$ implies that $f$ is the identity map. Therefore $|G|=|G\cdot h|$ for each $h\in H(\Gamma)$. Since for each half-edge $h$ of $\Gamma$, $h$ and $\nu(h)$ are incident to the same vertex, the permutation group $\ang{\nu}_0$ of $V(\Gamma)$ is trivial. Then the vertices of $\Gamma/G$ are in one-to-one correspondence with the vertices of $\Ga$. For any vertex $v\in V(\Ga/G)$, the multiplicity of the vertex $v\in V(\Ga)=V(\Ga/G)$ is defined as $\bm(v)|G|$. Note that $\bm(v)|G|=\bm(v)|G\cdot h|=\frac{\bm(v)\val(v)}{(\bm(v)\val(v),\val(v))}$ is a positive integer since $\bm(v)\val(v)$ is a positive integer, where $(\bm(v)\val(v),\val(v))$ is the greatest common divisor of $\bm(v)\val(v)$ and $\val(v)$, and $(\bm(v)\val(v),\val(v))$ is also the valency of $v$ in $\Gamma/G$. Moreover, the orientation of $\Ga/G$ is induced from the orientation of $\Ga$. Therefore the half-edged graph $\Gamma/G$ becomes an ${\bf mBG}$.
\end{proof}

We give an example to illustrate the above quotient $\Ga/\ang{\nu}$.

\begin{Ex} \label{Ex-modified-Brauer-tree}
	Let $\Ga$ be the ${\bf bFBG}$ in Example \ref{Example-equi-def}. We have that $\Ga/\langle\nu\rangle$ is a multiplicity-free ${\bf mBG}$ given by the diagram
	$$\begin{tikzpicture}
		\draw (-2,0) -- (-0.2,0);
		\draw (0.2,0) -- (3.8,0);
		\node at(0,0) {$a$};
		\node at(4,0) {$b$};
		\node at(-1,0.2) {$\ang{\nu}\cdot(1,a)$};
		\node at(1,0.2) {$\ang{\nu}\cdot(2,a)$};
		\node at(3,0.2) {$\ang{\nu}\cdot(2,b)$};
		\node at(4.5,-0.1) {.};
	\end{tikzpicture}$$
	Note that $\ang{\nu}\cdot(1,a)$ is an isolated half-edge in $\Ga/\langle\nu\rangle$, the cyclic ordering of the edges incident to vertex $a$ is given by $\ang{\nu}\cdot(1,a)<\ang{\nu}\cdot(2,a)<\ang{\nu}\cdot(1,a)$, to vertex $b$ is given by $\ang{\nu}\cdot(2,b)<\ang{\nu}\cdot(2,b)$.
\end{Ex}

Note that if $\Ga$ is a ${\bf bFBG}$ then there is a covering $\Ga\rightarrow\Ga/\ang{\nu}$ of Brauer $G$-sets (see \cite{LL3}). However, in the present paper we only use coverings at graphs level (see Section \ref{sec-rep-type-graded-bFBGA}), so we do not pursue this more general covering theory here.

\begin{Def}[see {\cite[Definition 2.33]{LL3}}]\label{Def-modi-tree}
	Let $\Ga=(V(\Ga),E(\Ga),\bm,\mathfrak{o})$ be a ${\bf bFBG}$ and $\langle\nu\rangle$ the group of automorphisms generated by Nakayama automorphism $\nu$ of $\Ga$. Denote by $O(\Ga/\ang{\nu})$ the set of isolated half-edges of $\Ga/\ang{\nu}$. If $O(\Ga/\ang{\nu})=\varnothing$, then $\Ga/\ang{\nu}$ is a ${\bf BG}$ and define $R_{\Gamma}$ to be $\Ga/\ang{\nu}$; if $O(\Ga/\ang{\nu})\neq\varnothing$, then define $R_{\Gamma}$ to be the ${\bf BG}$ obtained by gluing the corresponding isolated half-edges in the two copies of $\Ga/\ang{\nu}$ together to form an edge. In any case, $R_{\Gamma}$ is called the reduced form of $E$.
\end{Def}

If a ${\bf bFBG}$ $\Ga$ is as in Example \ref{Example-BT}, then we have the reduced form $R_{\Gamma}$ as follows.

\begin{Ex}\label{Example-reduced-graph}
	\begin{enumerate}[$(1)$]
		\item In Example \ref{Example-BT} $(1)$, we have that $\Ga/\ang{\nu}$ is the following multiplicity-free ${\bf mBG}$
		$$\xymatrix@R-2em@L-1em{
			\ar@{-}@<2pt>[r]_{1}	a \ar@{-}@<-2pt>[r]^{2} &
		}. $$	
		Then $R_{\Gamma}$ is the following multiplicity-free ${\bf BG}$.
		$$\begin{small}\xymatrix@R-2em@L-1em{
				\ar@{-}@<2pt>[r]_{1}	a \ar@{-}@<-2pt>[r]^{2} &a'
			}
		\end{small} $$
		\item In Example \ref{Example-BT} $(2)$,
		we have that $\Ga/\ang{\nu}$ is the following multiplicity-free ${\bf mBG}$
		$$\xymatrix@R-2em@L-1em{
			&	& & e	\\
			&	a \ar@{-}[r]_{2} \ar@{-}[l]^{1} & b \ar@{-}[ur]_{7}\ar@{-}[dr]^{4}\\
			&& &c~. } $$	
		Then $R_{\Gamma}$ is the following multiplicity-free ${\bf BT}$.	
		$$\begin{small}\xymatrix@R-2em@L-1em{
				&	& & e	\\
				&	a \ar@{-}[r]_{2} \ar@(l,l)@{-}[ddd]^{1} & b \ar@{-}[ur]_{7}\ar@{-}[dr]^{4}\\
				&& & c
				&&& \\
				&	& & e'	\\
				&	a' \ar@{-}[r]_{2'} & b' \ar@{-}[ur]_{7'}\ar@{-}[dr]^{4'}\\
				&& &c'  }
		\end{small} $$
		\item In Example \ref{Example-BT} $(3)$,
		we have that $\Ga/\ang{\nu}$ is the following ${\bf BT}$ $$\xymatrix{ 	a \ar@{-}[r]^{1} 	& b \ar@{-}[r]^{3} & c 	}, $$
		where the multiplicity function $\bm$ of ${\bf BT}$ $\Ga/\ang{\nu}$ is given by $\bm(a)=3, \bm(b)=1, \bm(c)=1$. Note that $R_{\Gamma}=\Ga/\ang{\nu}$.
	\end{enumerate}
\end{Ex}

\begin{Def}\label{def-brauer-tree-with-orbifolds}
	Let $\Ga=(V(\Ga),E(\Ga),O(\Ga),\bm,\mathfrak{o})$ be an ${\bf mBG}$. If the remaining is a Brauer tree after removing all isolated half-edges from $\Gamma$, then $\Gamma$ is called a modified Brauer tree (abbr. ${\bf mBT}$).
\end{Def}

The quotients $\Ga/\ang{\nu}$ in Example \ref{Ex-modified-Brauer-tree} and in Examples \ref{Example-reduced-graph} are ${\bf mBT}$s. In fact, for a ${\bf bFBG}$ $\Gamma$, if the corresponding ${\bf bFBGA}$ $A_{\Gamma}$ is representation-finite or representation-infinite domestic, then the quotient $\Ga/\ang{\nu}$ is either a ${\bf BG}$ or a multiplicity-free ${\bf mBT}$ (see Theorem \ref{Thm-finite-f-BG} and Theorem \ref{Thm-domestic-f-BG}).

It is well-known that a ${\bf BGA}$ is representation-finite if and only if its defining ${\bf BG}$ is a ${\bf BT}$. Bocian and Skowro\'{n}ski have described the defining ${\bf BG}$ of a domestic ${\bf BGA}$ (see \cite{BS} or \cite{S}). We call a ${\bf BG}$ {\it domestic} if its ${\bf BGA}$ is representation-infinite domestic. It is convenient to introduce the following notions.

\begin{Def}\label{def-rep-finite-domestic-modified-Brauer-tree}
	Let $\Ga=(V(\Ga),E(\Ga),O(\Ga),\bm,\mathfrak{o})$ be an ${\bf mBT}$. Denote by $|O(\Ga)|$ the cardinality of $O(\Ga)$. Then
	\begin{enumerate}[$(1)$]
		\item $\Ga$ is called representation-finite if $\Ga$ is multiplicity-free and $|O(\Ga)|=1$.
		\item $\Ga$ is called domestic if $\Ga$ is multiplicity-free and $|O(\Ga)|=2$.
	\end{enumerate}
\end{Def}

For any ${\bf bFBG}$, the characterization of the representation type of the corresponding ${\bf bFBGA}$ is given as follows.

\begin{Thm}[see {\cite[Lemma 3.8]{LL3}}] \label{Thm-finite-f-BG}
	Let $\Ga$ be a ${\bf bFBG}$ and $A$ the corresponding ${\bf bFBGA}$.  $A$ is of finite representation type if and only if $\Ga/\ang{\nu}$ belongs to one of the following two cases, where $\langle\nu\rangle$ is the group of automorphisms generated by Nakayama automorphism $\nu$ of $\Ga$.
	\begin{enumerate}[$(1)$]
		\item $\Ga/\ang{\nu}$ is a ${\bf BT}$.
		\item $\Ga/\ang{\nu}$ is a representation-finite ${\bf mBT}$.
	\end{enumerate}
\end{Thm}

\begin{Rem}
	According to the description of AR-components of representation-finite ${\bf bFBGA}$ in \cite[Section 4]{LL3}, we have the following for the above two cases.
	
	\begin{enumerate}[$(1)$]
		\item In the first case, the corresponding ${\bf bFBGA}$ is stably equivalent to some self-injective Nakayama algebra.
		\item In the second case, the corresponding ${\bf bFBGA}$ is stably equivalent to some canonical M\"obius algebra named by Asashiba in \cite{A2}.
	\end{enumerate}
\end{Rem}

\begin{Thm}[see {\cite[Lemma 3.10]{LL3}}]\label{Thm-domestic-f-BG}
	Let $\Ga$ be a ${\bf bFBG}$ and $A$ the corresponding ${\bf bFBGA}$.  $A$ is representation-infinite domestic if and only if $\Ga/\ang{\nu}$ belongs to one of the following two cases, where $\langle\nu\rangle$ is the group of automorphisms generated by Nakayama automorphism $\nu$ of $\Ga$.
	\begin{enumerate}[$(1)$]
		\item $\Ga/\ang{\nu}$ is a domestic ${\bf BG}$.
		\item $\Ga/\ang{\nu}$ is a domestic ${\bf mBT}$.
	\end{enumerate}
\end{Thm}

\begin{Rem}
	\begin{enumerate}[$(1)$]
		\item If $\Ga/\ang{\nu}$ belongs to the second case in Theorem \ref{Thm-finite-f-BG} (resp. Theorem \ref{Thm-domestic-f-BG} ), then the reduced form $R_{\Gamma}$ is a multiplicity-free ${\bf BT}$ (resp. multiplicity-free ${\bf BG}$ with a unique cycle of even length).
		\item According to Theorems \ref{Thm-finite-f-BG} and \ref{Thm-domestic-f-BG}, the ${\bf bFBGA}$ $A_{\Ga}$ in Example \ref{Example-reduced-graph} $(1)$ is representation-infinite domestic (in fact it is $2$-domestic by Appendix \ref{appendix} or by Section \ref{rep.type-sba}), and the ${\bf bFBGA}$s $A_{\Ga}$ in Example \ref{Example-reduced-graph} $(2)$ and Example \ref{Example-reduced-graph} $(3)$ are representation-finite.
	\end{enumerate}
\end{Rem}

\section{Representation types of the associated graded algebras of biserial fractional Brauer graph algebras}\label{sec-rep-type-graded-bFBGA}

In this section, we assume that $A=\bk Q/I$ is a ${\bf bFBGA}$ associated with a ${\bf bFBG}$ $\Ga$ and let $\gr(A)$ be the associated graded algebra with the radical filtration of $A$. Moreover, let $\overline{A}$ and $\overline{\gr(A)}$ be the associated string algebras of $A$ and $\gr(A)$ respectively. We have the following remark about the relationship between $\overline{A}$ and $\overline{\gr(A)}$.

\begin{Rem} [compare to {\cite[Lemma 2.9]{GL}}]
Denote by
	 $$\mathbb{L}=\{i\in Q_{0}|\rad(Ae_{i})/\soc(Ae_{i})=V_{i,1}\oplus V_{i,2}, V_{i,1}\neq 0, V_{i,2}\neq 0\},$$
	\begin{equation*}\label{equ-W}
		\mathbb{W}=\{i\in Q_{0}\vert \rad(A e_{i})/\soc(A e_{i})=V_1 \oplus V_2, V_1\neq0, V_2\neq0,\ell(V_1)\neq\ell(V_2)\}\subseteq \mathbb{L},
	\end{equation*}
	and
	$$\mathbb{L}'=\mathbb{L}-\mathbb{W}.$$
	Then the three algebras  $\overline{A}$, $\gr(A)$, $\overline{\gr(A)}$ are given by the same quiver $Q$ together with ideals $I_1$, $I'$, $I_2$ in $\bk Q$ respectively, where the three ideals $I_1$, $I'$, $I_2$ are defined as follows (here we use notations from Subsection \ref{sec-bfbga}).

	Let
	$$R_1:=\{\mbox{Relation of the first type in }I\},\quad I_0:=\langle \mbox{Relation of the second type or the third type in }I\rangle;$$
	The ideals $I_1$, $I'$, $I_2$ in $\bk Q$ are defined by
	$$I_1:=I_0+\langle p_i,q_i\mid i\in \mathbb{L}, p_i-q_i\in R_1\rangle;$$
	$$I':=I_0+\langle p_i-q_i\in R_1 \mid i\in \mathbb{L}'\rangle+\langle q_i\mid i\in \mathbb{W}, p_i-q_i\in R_1,\mbox{ the length of } q_i\mbox{ is less than the length of }p_i\rangle;$$
	$$I_2:=I_0+\langle p_i,q_i\mid i\in \mathbb{L}', p_i-q_i\in R_1\rangle+\langle q_i\mid i\in \mathbb{W}, p_i-q_i\in R_1, \mbox{ the length of } q_i\mbox{ is less than the length of }p_i\rangle;$$
	In particular, $\overline{A}$ is a quotient of the algebra $\overline{\gr(A)}$.
\end{Rem}

According to Corollary \ref{cor-induced-covering} and Theorem \ref{Theorem}, we have  the following result.

\begin{Prop}\label{main-cor}
	Let $A=\bk Q_A/I_A$ and $B=\bk Q_B/I_B$ be two ${\bf bFBGA}$s associated with some
${\bf bFBG}$s,
and let $\gr(A)=\bk Q_A/I_A'$ and $\gr(B)=\bk Q_B/I_B'$ be the associated graded algebras of $A$ and $B$ respectively. If there is a Galois covering of quivers with relations from  $(Q_A, I_A)$ to $(Q_{B},I_{B})$, then $\gr(A)$ and $\gr(B)$ have the same representation type.
\end{Prop}


For a ${\bf bFBG}$ $\Ga=(V(\Ga),E(\Ga),\bm,\mathfrak{o})$, an {\it automorphism} of $\Ga$ is a graph automorphism of $\Ga$ which preserves multiplicities of vertices and preserves orientation. Note that the Nakayama automorphism of $\Ga$ is an automorphism of $\Ga$ in the sense as above. A group of automorphisms $G$ of $\Ga$ is said to be {\it admissible} if there does not exist an automorphism $g\in G$ which maps a half-edge $h$ in $\Ga$ to the other half-edge $h'$ belonging to the same edge.

\begin{Def}[see {\cite[Definitions 3.1 and 3.30]{LL2}}]
Let $\Ga=(V(\Ga),E(\Ga),\bm,\mathfrak{o})$ and $\Ga'=(V(\Ga'),E(\Ga'),\bm',\mathfrak{o}')$ be two ${\bf bFBG}$s with the sets of half-edges $H(\Ga)$ and $H(\Ga')$ respectively, and let $f:\Ga\rightarrow \Ga'$ be a morphism of graphs. Then $f$ is said to be a covering of ${\bf bFBG}$s if it satisfies the following two conditions.
\begin{enumerate}[$(1)$]
		\item $f$ preserves cyclic ordering of half-edges, that is, if the half-edge $h_2$ is the successor of the half-edge $h_1$ under the cyclic ordering of half-edges incident to vertex $v$ in $\Ga$, then $f(h_2)$ is the successor of $f(h_1)$ under the cyclic ordering of half-edges incident to vertex $f_0(v)$ in $\Ga'$, where $f_0:V(\Ga)\rightarrow V(\Ga')$ is the map between the sets of vertices induced by $f$.
		\item $\bm(v)\val(v)=\bm'(f_0(v))\val(f_0(v))$ for each vertex $v$ in $\Gamma$, where $\val(v)$ (resp. $\val(f_0(v))$) denotes the valency of the vertex $v$ (resp. $f_0(v)$).
	\end{enumerate}
A covering $f:\Ga\rightarrow \Ga'$ is called a Galois covering of ${\bf bFBG}$s if there exists an admissible group of automorphisms $G$ of $\Ga$ such that $fg=f$ as maps from $H(\Ga)$ to $H(\Ga')$ for each $g\in G$, and $f^{-1}(h')$ is a transitive $G$-set for each $h'\in H(\Ga')$.
\end{Def}

Let $A=\bk Q_A/I_A$ be a ${\bf bFBGA}$ associated with some
 ${\bf bFBG}$ $\Ga$. Let $\nu$ be the Nakayama automorphism of $\Ga$. If the group $\ang{\nu}$ is admissible, then $\Ga/\ang{\nu}$ is a ${\bf BG}$ and there is a Galois covering $\Ga\rightarrow \Ga/\ang{\nu}$ (cf. \cite[Lemma 3.36]{LL2}), which induces a Galois covering of quivers with relations from $(Q_A, I_A)$ to $(Q_{B},I_{B})$ (cf. \cite[Lemma 5.12]{LL2}), where $B=\bk Q_B/I_B$ is the ${\bf BGA}$ associated with $\Ga/\ang{\nu}$. If $\ang{\nu}$ is not admissible, then $\Ga/\ang{\nu}$ is not a ${\bf BG}$ (that is, it is an ${\bf mBG}$ whose isolated half-edge set is nonempty) and according to the proof of \cite[Proposition 5.12]{LL3}, we have the following lemma.

\begin{Lem} \label{two-coverings}
Let $\Ga$ be a ${\bf bFBG}$ with Nakayama automorphism $\nu$ such that $\Ga/\ang{\nu}$ is not a ${\bf BG}$, and let $R_{\Gamma}$ be the reduced form of $\Ga$. Then there exists a ${\bf bFBG}$ $\Ga'$ together with two Galois coverings $p:\Ga'\rightarrow \Ga$ and $q:\Ga'\rightarrow R_{\Gamma}$.
\end{Lem}

The above ${\bf bFBG}$ $\Ga'$ is constructed as follows. Denote by $\Ga=(V(\Ga),E(\Ga),\bm,\mathfrak{o})$ and denote by $H(\Ga)$ the set of its half-edges. The set of vertices $V(\Ga')$ (resp. the set of half-edges $H(\Ga')$) of $\Ga'$ is $V(\Ga)_1\sqcup V(\Ga)_2$ with $V(\Ga)_1=V(\Ga)_2=V(\Ga)$ (resp. $H(\Ga)_1\sqcup H(\Ga)_2$ with $H(\Ga)_1=H(\Ga)_2=H(\Ga)$). For a vertex $v\in V(\Ga)$ (resp. a half-edge $h\in H(\Ga)$), we denote by $v_i$ (resp. $h_i$) the vertex $v$ (resp. half-edge $h$) in $V(\Ga)_i$ (resp. $H(\Ga)_i$) for $i=1,2$. The edges of $\Ga'$ are defined as follows. For a half-edge $h\in H(\Ga)$, suppose that $h$ together with $h'\in H(\Gamma)$ form an edge. If $h'$ belongs to the $\ang{\nu}$-orbit of $h$, then $h_i$ together with $h'_{3-i}$ form an edge of $\Ga'$ for $i=1,2$; otherwise, $h_i$ together with $h'_i$ form an edge of $\Ga'$ for $i=1,2$. The multiplicity $\bm'$ of $\Ga'$ is given by $\bm'(v_i)=\bm(v)$ for each $v_i\in V(\Ga)_i$ with $i=1,2$. The orientation $\mathfrak{o}'$ of $\Ga'$ is defined as follows. If the cyclic ordering of the half-edges incident to vertex $v\in V(\Ga)$ is given by $h^{(1)}<h^{(2)}<\cdots<h^{(n)}<h^{(1)}$, then the cyclic ordering of the half-edges incident to vertex $v_i\in V(\Ga)_i$ is given by $h^{(1)}_{i}<h^{(2)}_{i}<\cdots<h^{(n)}_{i}<h^{(1)}_{i}$ for $i=1,2$.

\begin{Ex}\label{Example-two-coverings}
Let $\Ga$ be the ${\bf bFBG}$ in Example \ref{Example-BT} $(1)$. Write the edge of $\Ga$ by $i=(i',i'')$ for $i=1,2$. Then $\Ga'$ is the ${\bf bFBG}$
$$\tikzset{every picture/.style={line width=0.75pt}}
\begin{tikzpicture}[x=15pt,y=15pt,yscale=1,xscale=1]
\node at(0,0) {$a_1$};
\node at(5,0) {$a_2$};
\draw    (-0.3,0.1) .. controls (-4.5,3) and (0.5,3) .. (4.7,0.1) ;
\draw    (-0.2,0.3) .. controls (-1.5,3) and (3.5,3) .. (4.8,0.3) ;
\draw    (0.2,0.3) .. controls (1.5,3) and (6.5,3) .. (5.2,0.3) ;
\draw    (0.3,0.1) .. controls (4.5,3) and (9.5,3) .. (5.3,0.1) ;
\end{tikzpicture}$$
where $\bm(a_1)=\bm(a_2)=\frac{1}{2}$ and the four edges from left to right are given by $(1_1',1_2''),(2_1',2_2''),(1_1'',1_2'),(2_1'',2_2')$. Note that both Galois coverings $p$ and $q$ in this example have the same group $\mathbb{Z}/2\mathbb{Z}$.
\end{Ex}

Combining Theorem \ref{Theorem}, Proposition \ref{main-cor}, Lemma \ref{two-coverings} and \cite[Lemma 5.12]{LL2}, we have the following theorem.

\begin{Thm}\label{second-covering}
Let $A=\bk Q_A/I_A$ be a ${\bf bFBGA}$ associated with some
 ${\bf bFBG}$  $\Ga$. Suppose that $\Ga/\ang{\nu}$ is not a ${\bf BG}$, $R_{\Gamma}$ is the reduced form of $\Ga$ and $B=\bk Q_B/I_B$ is the ${\bf BGA}$ associated with $R_{\Gamma}$. Then there exists a (finite-dimensional indecomposable) ${\bf bFBGA}$ $C=\bk Q_C/I_C$ such that there are two Galois coverings of quivers with relations from $(Q_C, I_C)$ to $(Q_{B},I_{B})$ and from $(Q_C, I_C)$ to $(Q_{A},I_{A})$. In particular, $\gr(A)$ and $\gr(B)$ have the same representation type.
\end{Thm}

We now prove our main result in this section about the representation type of the associated graded algebras of ${\bf bFBGA}$s.  

\begin{Thm} \label{Thm-rep-typ}
	Let $A$ be a ${\bf bFBGA}$ associated with some ${\bf bFBG}$  $\Ga$ and $\mathrm{gr}(A)$ the graded algebra associated with the radical filtration of $A$. Suppose that $\nu$ is the Nakayama automorphism of $\Ga$ and $R_{\Gamma}$ is the reduced form of $\Ga$. Let $B$ be the ${\bf BGA}$ associated with $R_{\Gamma}$ and $\gr(B)$  the graded algebra associated with the radical filtration of $B$.  		
	Then we have the following.
\begin{enumerate}[$(1)$]
\item If $\Ga/\ang{\nu}$ is one of the four cases: (i) a ${\bf BT}$; (ii) a representation-finite ${\bf mBT}$; (iii) a domestic ${\bf BG}$; (iv) a domestic ${\bf mBT}$, then
	\begin{enumerate}
		\item [$(1.1)$] $\gr(A)$ is of finite representation type if and only if $\gr(B)$ is of finite representation type.
		\item [$(1.2)$] the following are equivalent:
		\begin{enumerate}
			\item [$(1.2.1)$] 	$\gr(A)$ is domestic.
			\item [$(1.2.2)$]  $\gr(A)$ is of polynomial growth.
			\item [$(1.2.3)$] $\gr(B)$ is  domestic.
			\item [$(1.2.4)$]  $\gr(B)$ is of polynomial growth.
		\end{enumerate}
\end{enumerate}
\item If $\Ga/\ang{\nu}$ does not belong to the four cases in $(1)$, then both $\gr(A)$ and $\gr(B)$ are nondomestic and not of polynomial growth.
\end{enumerate}
\end{Thm}

\begin{proof}
Since $A$ and $\overline{A}$ (resp. $\gr(A)$ and $\overline{\gr(A)}$) have the same representation type, and since $\overline{A}$ is a quotient of the algebra $\overline{\gr(A)}$, if $A$ is nondomestic (resp. $A$ is not of polynomial growth), then $\gr(A)$ is nondomestic (resp. $\gr(A)$ is not of polynomial growth). By Theorem \ref{Thm-finite-f-BG} and Theorem \ref{Thm-domestic-f-BG}, there are two situations to be considered.

If $\Ga/\ang{\nu}$ is a ${\bf BG}$, then by  \cite[Lemma 3.36 and Lemma 5.12]{LL2} and Proposition \ref{main-cor}, $\gr(A)$ is of finite representation type (resp. domestic) if and only if $\gr(B)$ is of finite representation type (resp. domestic).

If $\Ga/\ang{\nu}$ is a representation-finite or domestic ${\bf mBT}$, then by Theorem \ref{second-covering} and Proposition \ref{main-cor}, there is a ${\bf bFBGA}$ $C$ such that $\gr(A)$ or $\gr(B)$ is of finite representation type (resp. domestic) if and only if $\gr(C)$ is of finite representation type (resp. domestic).

In either case, $\gr(A)$ is of finite representation type (resp. domestic) if and only if $\gr(B)$ is of finite representation type (resp. domestic). Moreover, if $\Ga/\ang{\nu}$ does not belong to the above four cases, then both $A$ and $B$ are nondomestic and therefore both $\gr(A)$ and $\gr(B)$ are nondomestic.

Finally, by Theorem \ref{Theorem}, we have that $\gr(A)$ (resp. $\gr(B)$) is domestic if and only if $\gr(A)$ (resp. $\gr(B)$) is of polynomial growth.
\end{proof}

We illustrate the above theorem by the following example.

\begin{Ex}
\begin{enumerate}[$(1)$]
	\item If $\Ga$ is as in Example \ref{Example-BT} $(1)$, then $A$ is isomorphic to $\gr(A)$. By Theorem \ref{Thm-domestic-f-BG}  and Example \ref{Example-reduced-graph}, $A\cong\gr(A)$ is representation-infinite domestic.
	\item  If $\Ga$ is  as in Example \ref{Example-BT} $(2)$ (resp. Example \ref{Example-BT} $(3)$), then by Theorem \ref{Thm-finite-f-BG} and Example \ref{Example-reduced-graph}, $A$ is of finite representation type. By Example \ref{Example-reduced-graph}, \cite[Theorem 3.18]{GLY} (resp. \cite[Theorem 4.5]{GL}) and Theorem \ref{Thm-rep-typ}, we have that $\gr(A)$ is representation-infinite domestic (resp. $\gr(A)$  is of finite representation type).
\end{enumerate}
\end{Ex}

\section{A description of  representation type of graded biserial fractional Brauer graph algebra $\gr(A_{\Gamma})$ in terms of the quotient $\Ga/\ang{\nu}$}\label{sec-quotient}

The characterization of representation-finite and domestic graded ${\bf BGA}$s in terms of the associated ${\bf BG}$s was given in \cite{GL} and \cite{GLY} respectively. In this section, we similarly describe the representation type of graded ${\bf bFBGA}$ $\gr(A_{\Gamma})$ in terms of the quotient $\Ga/\ang{\nu}$.

We first generalize the notions of graded degree and unbalanced edge in ${\bf BG}$ to ${\bf bFBG}$ and ${\bf mBT}$.

\begin{Def}\label{graded-degree-modified-Brauer-tree}
	Let $\Ga=(V(\Ga),E(\Ga),O(\Ga),\bm,\mathfrak{o})$ be a ${\bf bFBG}$ or an ${\bf mBT}$. For each vertex $v$ in $\Ga$, we denote by $\val(v)$ the valency of $v$, that is, the number of half-edges (both isolated and non-isolated) incident to $v$. If $\bm(v)\val(v)=1$ such that $v$ is connected to another vertex $v'$ by an edge, then we define the graded degree $\grd(v)$ of $v$ to be $\bm(v')\val(v')$. Otherwise, we define the graded degree $\grd(v)$ of $v$ to be $\bm(v)\val(v)$.
\end{Def}

\begin{Def}\label{unbalanced-edge-modified-Brauer-tree} Let $\Ga$ be a ${\bf bFBG}$ or an ${\bf mBT}$ with graded degree function $\grd$. We call an edge $v_{1}\defeqi v_{2}$ in $\Ga$ with $ \grd(v_{1})\neq \grd(v_{2})$ an unbalanced edge, and denote the endpoints of an unbalanced edge $i$ by $v^{(i)}_L$, $v^{(i)}_S$ with $\grd(v^{(i)}_L)>\grd(v^{(i)}_S)$, whenever the context is clear, we will omit the superscript of $v^{(i)}_L$ or $v^{(i)}_S$.
\end{Def}

We now generalize some notions for a ${\bf BG}$ from \cite{GL} to the situation of an ${\bf mBT}$.

\begin{Def}[compare with {\cite[Definition 3.7]{GL}}]\label{walk-in-BG}
	Let $\Ga$ be a ${\bf BG}$ or an ${\bf mBT}$, and let $u,v$ be two distinct vertices in $\Ga$. A walk from $u$ to $v$ is a sequence $[v_{1},a_{1},v_{2},\ldots, v_{k-1}, a_{k-1},v_{k}]$ of vertices and edges, where $v_{1}=u$, $v_{k}=v$, $a_i$ is an edge incident to the vertices $v_i$ and $v_{i+1}$ for each $1\leq i\leq k-1$, and all vertices (and hence all edges) are pairwise distinct. The length of a walk is defined as the number of edges in this walk.
\end{Def}

Note that the above notion of walk should not be confused with the same notion defined in \cite[Subsection 3.2]{LL2}. Note also that if $\Ga$ is an ${\bf mBT}$, then for every two vertices $u,v$ in $\Ga$, there exists a unique walk $w$ from $u$ to $v$, and $d_{\Ga}(u,v)$ denotes the length of $w$ (that is, the number of edges in $w$).

We say that a walk $[v_1,a_1,v_2,...,v_{k-1},a_{k-1},v_k]$ is degree decreasing if $\grd(v_1)\geq\grd(v_2)\geq\cdots\geq\grd(v_k)$.

\begin{Def} [see {\cite[Definition 3.14]{GLY}}]\label{two-numbers}
	Let $\Ga$ be an ${\bf mBT}$ with an exceptional vertex $v_0$.
	\begin{enumerate}[$(1)$]
		\item Define $\kappa_0$ to be the number of unbalanced edges $v_{S}\defeqi v_{L}$ such that $d_{\Ga}(v_0,v_S)+1=d_{\Ga}(v_0,v_L)$.
		\item Define $\kappa_1$ to be  the number of unbalanced edge pairs in $\Ga$, where an unbalanced edge pair is a set  $\{i,j\}$ of two unbalanced edges $v_{S}^{(i)}\defeqi v_{L}^{(i)}$ and $v_{S}^{(j)}\defeqj v_{L}^{(j)}$ satisfying $d_{\Ga}(v_{S}^{(i)},v_{S}^{(j)})+2=d_{\Ga}(v_{L}^{(i)},v_{L}^{(j)})$.
	\end{enumerate}
\end{Def}

We recall from \cite{GL} and \cite{GLY} the description of ${\bf BGA}$s whose associated graded algebras are representation-finite or domestic. In the following of this paper when we say domestic or polynomial growth we always mean representation-infinite.

\begin{Thm} [see {\cite[Theorem 4.5]{GL}}]\label{grA-finite-type}
	Let $A$ be the ${\bf BGA}$ associated with a ${\bf BG}$ $\Ga$ and $\mathrm{gr}(A)$ the graded algebra associated with the radical filtration of $A$. Then $\gr(A)$ is representation-finite if and only if the following two conditions are satisfied:
	\begin{enumerate}[$(1)$]
		\item $\Ga$ is a ${\bf BT}$ with an exceptional vertex $v_0$ of multiplicity $\bm_0$;
		\item Denote by $v_1$ the exceptional vertex $v_0$ when $\bm_0>1$ or one of the vertices with maximal graded degree when $\bm_0=1$. Then any walk from $v_1$ is degree decreasing.
	\end{enumerate}
\end{Thm}

\begin{Thm} [see {\cite[Theorem 3.18]{GLY}}]\label{grA-infinite-type}
	Let $A$ be the ${\bf BGA}$ associated with a ${\bf BG}$ $\Ga=(V(\Ga),E(\Ga),\bm)$ and $\mathrm{gr}(A)$ the graded algebra associated with the radical filtration of $A$, where $V(\Ga)$ is the vertex set, $E(\Ga)$ is the edge set and $\bm$ is the multiplicity function of $\Ga$. Then $\gr(A)$ is domestic if and only if $\gr(A)$ is of polynomial growth. Moreover,
	\begin{enumerate}[$\bullet$]
		\item $\gr(A)$ is $1$-domestic if and only if one of the following holds:
		\begin{itemize}
			\item[$(1)$] $\Ga$ is a ${\bf BT}$ with an exceptional vertex $v_0$ of multiplicity $\bm_0$ such that $\kappa_0(\bm_0-1)+\kappa_1=1$.
			\item[$(2)$] $\Ga$ is a tree and there exist two distinct vertices $w_0$, $w_1$, such that the following conditions hold:
			\begin{itemize}
				\item[$(2.1)$] $\bm(w_0)=\bm(w_1)=2$ and $\bm(v)=1$ for $v\neq w_0,w_1$,
				\item[$(2.2)$] $\grd(w_0)=\grd(w_1)$,
				\item[$(2.3)$] Any walk from $w_0$ (or from $w_1$) is degree decreasing.
			\end{itemize}
			\item[$(3)$] $\Ga$ is a multiplicity-free ${\bf BG}$ with a unique cycle of odd length, and satisfies the
			following conditions:
			\begin{itemize}
				\item[$(3.1)$] $\grd(u)=\grd(v)$ for any two vertices $u$ and $v$ in the unique cycle,
				\item[$(3.2)$] Any walk from any vertex in the unique cycle is degree decreasing.
			\end{itemize}
		\end{itemize}
		\item  $\gr(A)$ is $2$-domestic if and only if $\Ga$ satisfies the following conditions:
		\begin{itemize}
			\item[$(1)$] $\Ga$ is a multiplicity-free ${\bf BG}$ with a unique cycle of even length,
			\item[$(2)$] $\grd(u)=\grd(v)$ for any two vertices $u$ and $v$ in the unique cycle,
			\item[$(3)$] Any walk from any vertex in the unique cycle is degree decreasing.
		\end{itemize}
		\item  $\gr(A)$ is not $n$-domestic for $n\geq 3$.
	\end{enumerate}
\end{Thm}

Combining the above results together with Theorem \ref{Thm-rep-typ}, we can describe the representation type of the associated graded algebras of ${\bf bFBGA}$s in terms of the quotient $\Ga/\ang{\nu}$ as follows.
\begin{Thm}\label{thm}
	Let $A$ be a ${\bf bFBGA}$ associated with some ${\bf bFBG}$  $\Ga$ and $\mathrm{gr}(A)$ the graded algebra associated with the radical filtration of $A$. Suppose that $\nu$ is the Nakayama automorphism of $\Ga$ and $\langle\nu\rangle$ is the group of automorphisms generated by  $\nu$.  		
	Then we have the following.
	\begin{enumerate}[$(1)$]
		\item 	$\gr(A)$ is of finite representation type if and only if one of the following holds.
		\begin{enumerate}
			\item[$(1a)$] $\Ga/\ang{\nu}$ is a ${\bf BT}$ with an exceptional vertex $v_0$ of multiplicity $\bm_0$ such that $\kappa_0(\bm_0-1)+\kappa_1=0$.
			\item[$(1b)$] $\Ga/\ang{\nu}$ is a representation-finite ${\bf mBT}$ (see Definition \ref{def-rep-finite-domestic-modified-Brauer-tree}) satisfying $\kappa_0=0$ and $\kappa_1=0$, in other words, $\Ga/\ang{\nu}$ is a representation-finite ${\bf mBT}$ such that any walk from $v_0$ is degree decreasing,  where $v_0$ is the exceptional vertex and it is also the vertex incident with the unique isolated half-edge in  $\Ga/\ang{\nu}$.
		\end{enumerate}
		\item 	$\gr(A)$ is domestic if and only if one of the following holds.	
		\begin{enumerate}[$(2a)$]
			\item $\Ga/\ang{\nu}$ is a ${\bf BT}$ with an exceptional vertex $v_0$ of multiplicity $\bm_0$ such that $\kappa_0(\bm_0-1)+\kappa_1=1$.
			\item $\Ga/\ang{\nu}$ is a representation-finite ${\bf mBT}$ satisfying $\kappa_0=1$ and $\kappa_1=0$, where the exceptional vertex is the vertex incident with the unique isolated half-edge in  $\Ga/\ang{\nu}$.
			\item $\Ga/\ang{\nu}$ is a domestic ${\bf mBT}$ (see Definition \ref{def-rep-finite-domestic-modified-Brauer-tree}) such that the following conditions hold:
			\begin{enumerate}[$(c.1)$]
				\item $\grd(w_0)=\grd(w_1)$, where $w_0$ and $w_1$ are the vertices incident with the exactly two isolated half-edges in $\Ga/\ang{\nu}$.
				\item Any walk from $w_0$ (or from $w_1$) is degree decreasing.
			\end{enumerate}
			\item  $\Ga/\ang{\nu}$ is a ${\bf BG}$ whose underlying graph is a tree and there exist two distinct vertices $w_0,w_1$, such that the following conditions hold:
			\begin{enumerate}[$(2d.1)$]
				\item $\bm(w_0)=\bm(w_1)=2$ and $\bm(v)=1$ for $v\neq w_0,w_1$, where $\bm$ is the multiplicity function of ${\bf BG}$ $\Ga/\ang{\nu}$.
				\item $\grd(w_0)=\grd(w_1)$.
				\item Any walk from $w_0$ (or from $w_1$) is degree decreasing.
			\end{enumerate}
			\item $\Ga/\ang{\nu}$ is a multiplicity-free ${\bf BG}$ with a unique cycle, and satisfies the following conditions:	
			\begin{enumerate}[$(2e.1)$]
				\item $\grd(u)=\grd(v)$ for any two vertices $u$ and $v$ in the unique cycle.
				\item  Any walk from any vertex in the unique cycle is degree decreasing.
			\end{enumerate}
		\end{enumerate}
	\end{enumerate}
\end{Thm}

\begin{proof}
	Denote by $B$ be the ${\bf BGA}$ associated with the reduced form $R_{\Gamma}$ of $\Ga$ and $\gr(B)$ the graded algebra associated with the radical filtration of $B$. It follows from Theorem \ref{Thm-rep-typ} that $\gr(A)$ is of finite representation type (resp. domestic) if and only if $\gr(B)$ is of finite representation type (resp. domestic).
	
	Suppose that $\gr(A)$ is representation-finite. It follows that $A$ is also representation-finite. According to Theorem \ref{Thm-finite-f-BG}, $\Ga/\ang{\nu}$ belongs to one of the following two cases:
	\begin{enumerate}[$(a)$]
		\item $\Ga/\ang{\nu}$ is a ${\bf BT}$.
		\item $\Ga/\ang{\nu}$ is a representation-finite ${\bf mBT}$.
	\end{enumerate}
	Since $\gr(B)$ is representation-finite, $R_{\Gamma}$ is a ${\bf BT}$ satisfying the two conditions in Theorem \ref{grA-finite-type}. In case $(a)$, since $\Ga/\ang{\nu}=R_{\Gamma}$, $\Ga/\ang{\nu}$ is a ${\bf BT}$ with an exceptional vertex $v_0$ of multiplicity $\bm_0$ such that $\kappa_0(\bm_0-1)+\kappa_1=0$. Then $\Ga/\ang{\nu}$ satisfies condition $(1a)$ in Theorem \ref{thm}. In case $(b)$, denote by $v_1$ the vertex that the unique isolated half-edge in $\Ga/\ang{\nu}$ is incident to. Since $R_{\Gamma}$ satisfies the second condition in Theorem \ref{grA-finite-type}, we have that the vertex $v_1$ has maximal graded degree, and any walk from $v_1$ is degree decreasing. So $\Ga/\ang{\nu}$ satisfies condition $(1b)$ in Theorem \ref{thm}.
	
	Conversely, if $\Ga/\ang{\nu}$ satisfies either condition $(1a)$ or condition $(1b)$ in Theorem \ref{thm}, then $R_{\Ga}$ is a ${\bf BT}$ with an exceptional vertex $v_0$ of multiplicity $\bm_0$ such that $\kappa_0(\bm_0-1)+\kappa_1=0$, and it follows from Theorem \ref{grA-finite-type} that $\gr(B)$ is of finite representation type. Therefore $\gr(A)$ is also of finite representation type.
	
	Now suppose that $\gr(A)$ is domestic. Then $A$ is representation-finite or domestic. According to Theorems \ref{Thm-finite-f-BG} and \ref{Thm-domestic-f-BG}, $\Ga/\ang{\nu}$ belongs to one of the following five cases:
	\begin{enumerate}[$(a)$]
		\item $\Ga/\ang{\nu}$ is a ${\bf BT}$.
		\item $\Ga/\ang{\nu}$ is a representation-finite ${\bf mBT}$.
		\item $\Ga/\ang{\nu}$ is a domestic ${\bf mBT}$.
		\item $\Ga/\ang{\nu}$ is a ${\bf BG}$ whose underlying graph is a tree, such that there are exactly two vertices with multiplicity $2$ and others with multiplicity $1$.
		\item $\Ga/\ang{\nu}$ is a multiplicity-free ${\bf BG}$ with a unique cycle.
	\end{enumerate}
	
	In case $(a)$ we have $R_{\Ga}=\Ga/\ang{\nu}$. Since $\gr(B)$ is domestic, according to Theorem \ref{grA-infinite-type}, $\Ga/\ang{\nu}$ is a ${\bf BT}$ with an exceptional vertex $v_0$ of multiplicity $\bm_0$ such that $\kappa_0(\bm_0-1)+\kappa_1=1$. Then $\Ga/\ang{\nu}$ satisfies condition $(2a)$ in Theorem \ref{thm}. In case $(b)$, since the reduced form $R_{\Gamma}$ of $\Ga$ is a multiplicity-free ${\bf BT}$, according to Theorem \ref{grA-infinite-type}, $R_{\Gamma}$ is a ${\bf BT}$ satisfying $\kappa_1=1$. Choose the vertex that the unique isolated half-edge in $\Ga/\ang{\nu}$ is incident to as the exceptional vertex in $\Ga/\ang{\nu}$, then $\Ga/\ang{\nu}$ is a representation-finite ${\bf mBT}$ such that $\kappa_0=1$ and $\kappa_1=0$. Then $\Ga/\ang{\nu}$ satisfies condition $(2b)$ in Theorem \ref{thm}.
	
	In case $(c)$, denote by $o_0$ and $o_1$ the exactly two isolated half-edges in $\Ga/\ang{\nu}$, and denote by $w_i$ the vertex that $o_i$ is incident to for $i=0,1$. Since $\gr(B)$ is domestic and since $R_{\Ga}$ is a multiplicity-free ${\bf BG}$ with a unique cycle of even length, by Theorem \ref{grA-infinite-type} we have that any two vertices in the unique cycle of $R_{\Ga}$ have the same graded degree, and any walk from any vertex in the unique cycle of $R_{\Ga}$ is degree decreasing. It follows that $\grd(w_0)=\grd(w_1)$ and any walk from $w_0$ (or from $w_1$) is degree decreasing. Therefore $\Ga/\ang{\nu}$ satisfies condition $(2c)$ in Theorem \ref{thm}.
	
	In case $(d)$, denote by $w_0$ and $w_1$ the exactly two vertices with multiplicity $2$ in $\Ga/\ang{\nu}$. Since $R_{\Ga}=\Ga/\ang{\nu}$ and since $\gr(B)$ is domestic, by Theorem \ref{grA-infinite-type}, $\grd(w_0)=\grd(w_1)$ and any walk from $w_0$ (or from $w_1$) is degree decreasing. Then $\Ga/\ang{\nu}$ satisfies condition $(2d)$ in Theorem \ref{thm}. In case $(e)$, since $R_{\Ga}=\Ga/\ang{\nu}$ and since $\gr(B)$ is domestic, according to Theorem \ref{grA-infinite-type}, $\Ga/\ang{\nu}$ is a multiplicity-free ${\bf BG}$ with a unique cycle such that $\grd(u)=\grd(v)$ for any two vertices $u$ and $v$ in the unique cycle and any walk from any vertex in the unique cycle is degree decreasing. Thus $\Ga/\ang{\nu}$ satisfies condition $(2e)$ in Theorem \ref{thm}.
	
	Conversely, if $\Ga/\ang{\nu}$ satisfies one of the conditions $(2a)$-$(2e)$ in Theorem \ref{thm}, then the reduced form $R_{\Ga}$ of $\Ga$ satisfies one of the conditions in Theorem \ref{grA-infinite-type}, which implies that $\gr(B)$ is domestic. Then $\gr(A)$ is also domestic.
\end{proof}

\appendix
\section{Construction of biserial fractional Brauer graph algebras whose associated graded algebras are representation-finite or domestic}\label{appendix}

In this appendix, we characterize representation-finite and domestic graded ${\bf bFBGA}$ $\gr(A_{\Gamma})$ in terms of the associated ${\bf bFBG}$ $\Ga$ (see Propositions \ref{rep-finite-(a)}-\ref{domestic-(e)}). Since we use the same method as in \cite[Section 5]{LL3}, we merely state these results without proofs. We will also determine the number $n$ if $\gr(A_{\Gamma})$ is $n$-domestic except when $\Ga/\ang{\nu}$ is a multiplicity-free ${\bf BG}$ with a unique cycle, where $\nu$ is the Nakayama automorphism of $\Ga$.

\subsection{Representation-finite case}\label{subsec-representation-finite-case}
\

In this subsection we suppose that $\Ga$ is a ${\bf bFBG}$ with Nakayama automorphism $\nu$ such that $\gr(A_{\Gamma})$ is representation-finite, and we denote by $r$ the order of $\nu$. It follows that $\Ga/\ang{\nu}$ satisfies either condition $(1a)$ or condition $(1b)$ in Theorem \ref{thm}. It can be shown that in both cases $\Ga$ is determined by $\Ga/\ang{\nu}$ together with the order $r$ of the Nakayama automorphism $\nu$ of $\Ga$ up to isomorphism. Now we construct $\Ga$ by using $\Ga/\ang{\nu}$ and $r$.

\medskip
{\it Case 1: $\Ga/\ang{\nu}$ satisfies condition $(1a)$ in Theorem \ref{thm}.}

Let $\Ga/\ang{\nu}=(V(\Ga/\ang{\nu}),E(\Ga/\ang{\nu}),\bm,\mathfrak{o})$, and let $v_0$ be the exceptional vertex in $\Ga/\ang{\nu}$ with multiplicity $\bm_0$. Then the order $r$ of the Nakayama automorphism $\nu$ of $\Ga$ satisfies $(r,\bm_0)=1$, where $(r,\bm_0)$ is the greatest common divisor of $r$ and $\bm_0$. Denote by $\Ga^{(a)}_{r}$ the ${\bf bFBG}$ defined as follows. The vertices set $V(\Ga^{(a)}_{r})$ of $\Ga^{(a)}_{r}$ is equal to $V(\Ga/\ang{\nu})$. The set of half-edges $H(\Ga^{(a)}_{r})$ of $\Ga^{(a)}_{r}$ is given by
$$H(\Ga^{(a)}_{r})=\{(h,\overline{i})\mid h\in H(\Ga/\ang{\nu}) \text{ and } \overline{i}\in\mathbb{Z}/r\mathbb{Z}\},$$
where $H(\Ga/\ang{\nu})$ is the set of half-edges of $\Ga/\ang{\nu}$. The half-edge $(h,\overline{i})$ is incident to the same vertex as the half-edge $h$. Two different half-edges $(h_1,\overline{i_1})$ and $(h_2,\overline{i_2})$ in $\Ga^{(a)}_{r}$ belong to the same edge if and only if $h_1$, $h_2$ are two different half-edges which belong to the same edge of $\Ga/\ang{\nu}$ and $\overline{i_1}=\overline{i_2}$. The multiplicity function $\bm':V(\Ga^{(a)}_{r})\rightarrow\mathbb{Q}_{>0}$ is defined by $\bm'(v)=\frac{\bm(v)}{r}$ for each $v\in V(\Ga^{(a)}_{r})=V(\Ga/\ang{\nu})$.

The orientation $\mathfrak{o}'$ of $\Ga^{(a)}_{r}$ is defined as follows. For a vertex $v\neq v_0$ in $\Ga^{(a)}_{r}$, suppose that in $\Ga/\ang{\nu}$ the cyclic ordering of the half-edges incident to vertex $v$ is given by $h_1<h_2<\cdots<h_n<h_1$. Then in $\Ga^{(a)}_{r}$ the cyclic ordering of the half-edges incident to vertex $v$ is given by
\begin{gather*}
(h_1,\overline{0})<(h_2,\overline{0})<\cdots<(h_n,\overline{0})<(h_1,\overline{\bm_0})
<(h_2,\overline{\bm_0})<\cdots<(h_n,\overline{\bm_0})
<(h_1,\overline{2\bm_0})<(h_2,\overline{2\bm_0})\\<\cdots<(h_n,\overline{2\bm_0})<\cdots  <(h_1,\overline{(r-1)\bm_0})<(h_2,\overline{(r-1)\bm_0})<\cdots<(h_n,\overline{(r-1)\bm_0})<(h_1,\overline{0}).
\end{gather*}
Suppose that in $\Ga/\ang{\nu}$ the cyclic ordering of the half-edges incident to vertex $v_0$ is given by $l_1<l_2<\cdots<l_k<l_1$.
Then in $\Ga^{(a)}_{r}$ the cyclic ordering of the half-edges incident to vertex $v_0$ is given by
\begin{gather*}
(l_1,\overline{0})<(l_2,\overline{0})<\cdots<(l_k,\overline{0})<(l_1,\overline{1})
<(l_2,\overline{1})<\cdots<(l_k,\overline{1})
<(l_1,\overline{2})<(1_2,\overline{2})<\cdots<(l_k,\overline{2})\\<\cdots  <(l_1,\overline{r-1})<(l_2,\overline{r-1})<\cdots<(l_k,\overline{r-1})<(l_1,\overline{0}).
\end{gather*}
\begin{Prop} \label{rep-finite-(a)}
Let $\Ga$ be a ${\bf bFBG}$ with $\Ga/\ang{\nu}$ satisfying condition $(1a)$ in Theorem \ref{thm}, and let $r$ be the order of the Nakayama automorphism $\nu$ of $\Ga$. Then $(r,\bm_0)=1$, where $(r,\bm_0)$ is the greatest common divisor of $r$ and $\bm_0$. Moreover, $\Ga$ is isomorphic to $\Ga^{(a)}_{r}$.
\end{Prop}

\medskip
{\it Case 2: $\Ga/\ang{\nu}$ satisfies condition $(1b)$ in Theorem \ref{thm}.}

According to \cite[Lemma 3.7]{LL3}, the order $r$ of the Nakayama automorphism $\nu$ of $\Ga$ is even. Let $\Ga/\ang{\nu}=(V(\Ga/\ang{\nu}),E(\Ga/\ang{\nu}),O(\Ga/\ang{\nu}),\bm,\mathfrak{o})$, and denote by $o$ the unique isolated half-edge in $\Ga/\ang{\nu}$. We construct a ${\bf bFBG}$ $\Ga^{(b)}_{r}$ as follows. The set of vertices $V(\Ga^{(b)}_{r})$ of $\Ga^{(b)}_{r}$ is equal to $V(\Ga/\ang{\nu})$, and the set of half-edges $H(\Ga^{(b)}_{r})$ of $\Ga^{(b)}_{r}$ is given by
$$H(\Ga^{(b)}_{r})=\{(h,\overline{i})\mid h\in H(\Ga/\ang{\nu}) \text{ and } \overline{i}\in\mathbb{Z}/r\mathbb{Z}\},$$
where $H(\Ga/\ang{\nu})$ is the set of half-edges (both isolated and non-isolated) of $\Ga/\ang{\nu}$. The half-edge $(h,\overline{i})$ is incident to the same vertex as the half-edge $h$. Two different half-edges $(h_1,\overline{i_1})$ and $(h_2,\overline{i_2})$ in $\Ga^{(b)}_{r}$ belong to the same edge if and only if $h_1$, $h_2$ are two different non-isolated half-edges which belong to the same edge of $\Ga/\ang{\nu}$ and $\overline{i_1}=\overline{i_2}$, or $h_1,h_2=o$ and $\overline{i_2}=\overline{i_1+\frac{r}{2}}$. The multiplicity function $\bm':V(\Ga^{(b)}_{r})\rightarrow\mathbb{Q}_{>0}$ is defined by $\bm'(v)=\frac{\bm(v)}{r}$ for each $v\in V(\Ga^{(b)}_{r})=V(\Ga/\ang{\nu})$.

The orientation $\mathfrak{o}'$ of $\Ga^{(b)}_{r}$ is defined as follows. For each vertex $v$ in $\Ga^{(b)}_{r}$, suppose that in $\Ga/\ang{\nu}$ the cyclic ordering of the half-edges incident to vertex $v$ is given by $h_1<h_2<\cdots<h_n<h_1$. Then in $\Ga^{(b)}_{r}$ the cyclic ordering of the half-edges incident to vertex $v$ is given by
\begin{gather*}
(h_1,\overline{0})<(h_2,\overline{0})<\cdots<(h_n,\overline{0})<(h_1,\overline{1})<(h_2,\overline{1})<\cdots<(h_n,\overline{1})<(h_1,\overline{2})<(h_2,\overline{2})<\cdots<(h_n,\overline{2}) \\<\cdots <(h_1,\overline{r-1})<(h_2,\overline{r-1})<\cdots<(h_n,\overline{r-1})<(h_1,\overline{0}).
\end{gather*}

\begin{Prop} \label{rep-finite-(b)}
Let $\Ga$ be a ${\bf bFBG}$ with $\Ga/\ang{\nu}$ satisfying condition $(1b)$ in Theorem \ref{thm}, and let $r$ be the order of the Nakayama automorphism $\nu$ of $\Ga$. Then $r$ is even, and $\Ga$ is isomorphic to $\Ga^{(b)}_{r}$.
\end{Prop}

\subsection{Domestic case}
\

In this subsection, we suppose that $\Ga$ is a ${\bf bFBG}$ with Nakayama automorphism $\nu$ such that $\gr(A_{\Gamma})$ is domestic. We denote by $r$ the order of $\nu$. It follows that $\Ga/\ang{\nu}$ satisfies one of the conditions $(2a)$-$(2e)$ in Theorem \ref{thm}.

\medskip
{\it Case 1: $\Ga/\ang{\nu}$ satisfies condition $(2a)$ in Theorem \ref{thm}.}

Denote by $\Ga^{(a)}_{r}$ the ${\bf bFBG}$  constructed in Subsection \ref{subsec-representation-finite-case}. Then we have

\begin{Prop} \label{domestic-(a)}
Let $\Ga$ be a ${\bf bFBG}$ with $\Ga/\ang{\nu}$ satisfying condition $(2a)$ in Theorem \ref{thm}, and let $r$ be the order of the Nakayama automorphism $\nu$ of $\Ga$. Then $(r,\bm_0)=1$, where $(r,\bm_0)$ is the greatest common divisor of $r$ and $\bm_0$. Moreover, $\Ga$ is isomorphic to $\Ga^{(a)}_{r}$.
\end{Prop}

Since $\Ga/\ang{\nu}$ satisfies condition $(2a)$ in Theorem \ref{thm}, it is a ${\bf BT}$ with an exceptional vertex $v_0$ of multiplicity $\bm_0$ such that $\kappa_0(\bm_0-1)+\kappa_1=1$. Therefore there are two subcases:

{\it Case 1.1: $\Ga/\ang{\nu}$ is a ${\bf BT}$ with an exceptional vertex $v_0$ of multiplicity $2$ such that $\kappa_0=1$ and $\kappa_1=0$.}

{\it Case 1.2: $\Ga/\ang{\nu}$ is a multiplicity-free ${\bf BT}$ with $\kappa_1=1$.}

For the representation type of $\gr(A_{\Gamma})$ we have the following proposition.

\begin{Prop} \label{domestic-(a1a2)}
Let $\Ga$ be a ${\bf bFBG}$ with $\Ga/\ang{\nu}$ satisfying condition $(2a)$ in Theorem \ref{thm}, and let $r$ be the order of the Nakayama automorphism $\nu$ of $\Ga$. Then we have the following.
\begin{enumerate}[$(1)$]
	\item If $\Ga/\ang{\nu}$ is a ${\bf BT}$ in Case 1.1, then $\gr(A_{\Gamma})$ is $1$-domestic.
	\item If $\Ga/\ang{\nu}$ is a ${\bf BT}$ in Case 1.2 and $n$ is even, then $\gr(A_{\Gamma})$ is $1$-domestic, where $\{i=v_{S}^{(i)}\defeqi v_{L}^{(i)},j=v_{S}^{(j)}\defeqj v_{L}^{(j)}\}$ is the unique unbalanced edge pair in $\Ga/\ang{\nu}$ and $n=d_{\Ga/\ang{\nu}}(v_{S}^{(i)},v_{S}^{(j)})$.
	\item If $\Ga/\ang{\nu}$ is a ${\bf BT}$ in Case 1.2 and $n$ is odd,
	then $\gr(A_{\Gamma})$ is $r$-domestic, where $\{i=v_{S}^{(i)}\defeqi v_{L}^{(i)},j=v_{S}^{(j)}\defeqj v_{L}^{(j)}\}$ is the unique unbalanced edge pair in $\Ga/\ang{\nu}$ and $n=d_{\Ga/\ang{\nu}}(v_{S}^{(i)},v_{S}^{(j)})$.
\end{enumerate}
\end{Prop}

\medskip
{\it Case 2: $\Ga/\ang{\nu}$ satisfies condition $(2b)$ in Theorem \ref{thm}.}

According to \cite[Lemma 3.7]{LL3}, the order $r$ of the Nakayama automorphism $\nu$ of $\Ga$ is even. We denote by $\Ga^{(b)}_{r}$ the ${\bf bFBG}$ constructed in Subsection \ref{subsec-representation-finite-case}. Then we have

\begin{Prop} \label{domestic-(b)}
Let $\Ga$ be a ${\bf bFBG}$ with $\Ga/\ang{\nu}$ satisfying condition $(2b)$ in Theorem \ref{thm}, and let $r$ be the order of the Nakayama automorphism $\nu$ of $\Ga$. Then $r$ is even, $\Ga$ is isomorphic to $\Ga^{(b)}_{r}$, and $\gr(A_{\Gamma})$ is $\frac{r}{2}$-domestic.
\end{Prop}

\medskip
{\it Case 3: $\Ga/\ang{\nu}$ satisfies condition $(2c)$ in Theorem \ref{thm}.}

According to \cite[Lemma 3.7]{LL3}, the order $r$ of the Nakayama automorphism $\nu$ of $\Ga$ is even. Let $\Ga/\ang{\nu}=(V(\Ga/\ang{\nu}),E(\Ga/\ang{\nu}),O(\Ga/\ang{\nu}),\bm,\mathfrak{o})$, and denote by $o_0$ and $o_1$ the exactly two isolated half-edges in $\Ga/\ang{\nu}$. We construct a ${\bf bFBG}$ $\Ga^{(c)}_{r}$ as follows (see \cite[Subsection 5.2]{LL3}). The set of vertices $V(\Ga^{(c)}_{r})$ of $\Ga^{(c)}_{r}$ is equal to $V(\Ga/\ang{\nu})$, and the set of half-edges $H(\Ga^{(c)}_{r})$ of $\Ga^{(c)}_{r}$ is given by
$$H(\Ga^{(c)}_{r})=\{(h,\overline{i})\mid h\in H(\Ga/\ang{\nu}) \text{ and } \overline{i}\in\mathbb{Z}/r\mathbb{Z}\},$$
where $H(\Ga/\ang{\nu})$ is the set of half-edges (both isolated and non-isolated) of $\Ga/\ang{\nu}$. The half-edge $(h,\overline{i})$ is incident to the same vertex as the half-edge $h$. Two different half-edges $(h_1,\overline{i_1})$ and $(h_2,\overline{i_2})$ in $\Ga^{(c)}_{r}$ belong to the same edge if and only if $h_1$, $h_2$ are two different non-isolated half-edges which belong to the same edge of $\Ga/\ang{\nu}$ and $\overline{i_1}=\overline{i_2}$, or $h_1,h_2=o_j$ for some $j=0,1$ and $\overline{i_2}=\overline{i_1+\frac{r}{2}}$. The multiplicity function $\bm':V(\Ga^{(c)}_{r})\rightarrow\mathbb{Q}_{>0}$ is defined by $\bm'(v)=\frac{\bm(v)}{r}$ for each $v\in V(\Ga^{(c)}_{r})=V(\Ga/\ang{\nu})$.

The orientation $\mathfrak{o}'$ of $\Ga^{(c)}_{r}$ is defined as follows. For each vertex $v$ in $\Ga^{(c)}_{r}$, suppose that in $\Ga/\ang{\nu}$ the cyclic ordering of the half-edges incident to vertex $v$ is given by $h_1<h_2<\cdots<h_n<h_1$. Then in $\Ga^{(c)}_{r}$ the cyclic ordering of the half-edges incident to vertex $v$ is given by
\begin{gather*}
	(h_1,\overline{0})<(h_2,\overline{0})<\cdots<(h_n,\overline{0})<(h_1,\overline{1})
	<(h_2,\overline{1})<\cdots<(h_n,\overline{1})<(h_1,\overline{2})<(h_2,\overline{2})<\cdots<(h_n,\overline{2})\\<\cdots  <(h_1,\overline{r-1})<(h_2,\overline{r-1})<\cdots<(h_n,\overline{r-1})<(h_1,\overline{0}).
\end{gather*}

\begin{Prop} \label{domestic-(c)}
Let $\Ga$ be a ${\bf bFBG}$ with $\Ga/\ang{\nu}$ satisfying condition $(2c)$ in Theorem \ref{thm}, and let $r$ be the order of the Nakayama automorphism $\nu$ of $\Ga$. Then $r$ is even, $\Ga$ is isomorphic to $\Ga^{(c)}_{r}$, and $\gr(A_{\Gamma})$ is $r$-domestic.
\end{Prop}

\medskip
{\it Case 4: $\Ga/\ang{\nu}$ satisfies condition $(2d)$ in Theorem \ref{thm}.}

By the proof of \cite[Lemma 5.6]{LL3}, the order $r$ of the Nakayama automorphism $\nu$ of $\Ga$ is odd. Let $\Ga/\ang{\nu}=(V(\Ga/\ang{\nu}),E(\Ga/\ang{\nu}),\bm,\mathfrak{o})$, and denote by $w_0,w_1$ the exactly two vertices in $\Ga/\ang{\nu}$ which have multiplicity $2$. We construct a ${\bf bFBG}$  $\Ga^{(d)}_{r}$ as follows (see \cite[Subsection 5.2]{LL3}). The set of vertices $V(\Ga^{(d)}_{r})$ of $\Ga^{(d)}_{r}$ is equal to $V(\Ga/\ang{\nu})$, and the set of half-edges $H(\Ga^{(d)}_{r})$ of $\Ga^{(d)}_{r}$ is given by
$$H(\Ga^{(d)}_{r})=\{(h,\overline{i})\mid h\in H(\Ga/\ang{\nu}) \text{ and } \overline{i}\in\mathbb{Z}/r\mathbb{Z}\},$$
where $H(\Ga/\ang{\nu})$ is the set of half-edges of $\Ga/\ang{\nu}$. The half-edge $(h,\overline{i})$ is incident to the same vertex as the half-edge $h$. Two different half-edges $(h_1,\overline{i_1})$ and $(h_2,\overline{i_2})$ in $\Ga^{(d)}_{r}$ belong to the same edge if and only if $h_1$, $h_2$ are two different half-edges which belong to the same edge of $\Ga/\ang{\nu}$ and $\overline{i_1}=\overline{i_2}$. The multiplicity function $\bm':V(\Ga^{(d)}_{r})\rightarrow\mathbb{Q}_{>0}$ is defined by $\bm'(v)=\frac{\bm(v)}{r}$ for each $v\in V(\Ga^{(d)}_{r})=V(\Ga/\ang{\nu})$.

The orientation $\mathfrak{o}'$ of $\Ga^{(d)}_{r}$ is defined as follows. For each vertex $v\neq w_0,w_1$ in $\Ga^{(d)}_{r}$, suppose that in $\Ga/\ang{\nu}$ the cyclic ordering of the half-edges incident to vertex $v$ is given by $h_1<h_2<\cdots<h_n<h_1$. Then in $\Ga^{(d)}_{r}$ the cyclic ordering of the half-edges incident to vertex $v$ is given by
\begin{gather*}
	(h_1,\overline{0})<(h_2,\overline{0})<\cdots<(h_n,\overline{0})<(h_1,\overline{1})
	<(h_2,\overline{1})<\cdots<(h_n,\overline{1})<(h_1,\overline{2})<(h_2,\overline{2})<\cdots<(h_n,\overline{2})\\<\cdots  <(h_1,\overline{r-1})<(h_2,\overline{r-1})<\cdots<(h_n,\overline{r-1})<(h_1,\overline{0}).
\end{gather*}
For vertex $v=w_0$ or $v=w_1$, suppose that in $\Ga/\ang{\nu}$ the cyclic ordering of the half-edges incident to vertex $v$ is given by $l_1<l_2<\cdots<l_k<l_1$. Then in $\Ga^{(d)}_{r}$ the cyclic ordering of the half-edges incident to vertex $v$ is given by
\begin{gather*}
(l_1,\overline{0})<(l_2,\overline{0})<\cdots<(l_k,\overline{0})<(l_1,\overline{\frac{r+1}{2}})
<(l_2,\overline{\frac{r+1}{2}})<\cdots<(l_k,\overline{\frac{r+1}{2}})
<(l_1,\overline{r+1})<(l_2,\overline{r+1})\\<\cdots<(l_k,\overline{r+1}) <\cdots <(l_1,\overline{\frac{(r-1)(r+1)}{2}})<(l_2,\overline{\frac{(r-1)(r+1)}{2}})<\cdots<(l_k,\overline{\frac{(r-1)(r+1)}{2}})<(l_1,\overline{0}).	
\end{gather*}

\begin{Prop} \label{domestic-(d)}
Let $\Ga$ be a ${\bf bFBG}$ with $\Ga/\ang{\nu}$ satisfying condition $(2d)$ in Theorem \ref{thm}, and let $r$ be the order of the Nakayama automorphism $\nu$ of $\Ga$. Then $r$ is odd, $\Ga$ is isomorphic to $\Ga^{(d)}_{r}$, and $\gr(A_{\Gamma})$ is $r$-domestic.
\end{Prop}

\medskip
{\it Case 5: $\Ga/\ang{\nu}$ satisfies condition $(2e)$ in Theorem \ref{thm}.}

Let $\Ga/\ang{\nu}=(V(\Ga/\ang{\nu}),E(\Ga/\ang{\nu}),\bm,\mathfrak{o})$, and denote by $r$ the order of the Nakayama automorphism of $\Ga$. Fix an edge $\{b,b'\}$ in $\Ga$ with $b$ and $b'$ its two half-edges. For each $l\in\{1,\cdots,r\}$, we construct a ${\bf bFBG}$  $\Ga^{(e)}_{rl}$ as follows.

The set of vertices $V(\Ga^{(e)}_{rl})$ of $\Ga^{(e)}_{r}$ is equal to $V(\Ga/\ang{\nu})$, and the set of half-edges $H(\Ga^{(e)}_{rl})$ of $\Ga^{(e)}_{rl}$ is given by
$$H(\Ga^{(e)}_{rl})=\{(h,\overline{i})\mid h\in H(\Ga/\ang{\nu}) \text{ and } \overline{i}\in\mathbb{Z}/r\mathbb{Z}\},$$
where $H(\Ga/\ang{\nu})$ is the set of half-edges of $\Ga/\ang{\nu}$. The half-edge $(h,\overline{i})$ is incident to the same vertex as the half-edge $h$. For two different half-edges $(h_1,\overline{i_1})$ and $(h_2,\overline{i_2})$ in $\Ga^{(e)}_{rl}$, if $h_1\neq b,b'$, then $(h_1,\overline{i_1})$ and $(h_2,\overline{i_2})$ belong to the same edge if and only if $h_1$, $h_2$ are two different half-edges which belong to the same edge of $\Ga/\ang{\nu}$ and $\overline{i_1}=\overline{i_2}$; if $h_1=b$, then $(h_1,\overline{i_1})$ and $(h_2,\overline{i_2})$ belong to the same edge if and only if $h_2=b'$ and $\overline{i_2}=\overline{i_1 +l}$; if $h_1=b'$, then $(h_1,\overline{i_1})$ and $(h_2,\overline{i_2})$ belong to the same edge if and only if $h_2=b$ and $\overline{i_2}=\overline{i_1 -l}$. The multiplicity function $\bm':V(\Ga^{(e)}_{rl})\rightarrow\mathbb{Q}_{>0}$ is defined by $\bm'(v)=\frac{\bm(v)}{r}$ for each $v\in V(\Ga^{(e)}_{rl})=V(\Ga/\ang{\nu})$.

The orientation $\mathfrak{o}'$ of $\Ga^{(e)}_{rl}$ is defined as follows. For each vertex $v$ in $\Ga^{(d)}_{r}$, suppose that in $\Ga/\ang{\nu}$ the cyclic ordering of the half-edges incident to vertex $v$ is given by $h_1<h_2<\cdots<h_n<h_1$. Then in $\Ga^{(e)}_{rl}$ the cyclic ordering of the half-edges incident to vertex $v$ is given by
\begin{gather*}
	(h_1,\overline{0})<(h_2,\overline{0})<\cdots<(h_n,\overline{0})<(h_1,\overline{1})
	<(h_2,\overline{1})<\cdots<(h_n,\overline{1})<(h_1,\overline{2})<(h_2,\overline{2})<\cdots<(h_n,\overline{2})\\<\cdots  <(h_1,\overline{r-1})<(h_2,\overline{r-1})<\cdots<(h_n,\overline{r-1})<(h_1,\overline{0}).
\end{gather*}

\begin{Prop} \label{domestic-(e)}
Let $\Ga$ be a ${\bf bFBG}$ with $\Ga/\ang{\nu}$ satisfying condition $(2e)$ in Theorem \ref{thm}, and let $r$ be the order of the Nakayama automorphism $\nu$ of $\Ga$. Then $\Ga$ is isomorphic to $\Ga^{(e)}_{rl}$ for some $l\in\{1,\cdots,r\}$.
\end{Prop}

\begin{Rem}  \label{remark-exceptional-case}
	\begin{enumerate}[$(1)$]
		\item  $\Ga$ is uniquely determined by $\Ga/\ang{\nu}$ and the order $r$ of the Nakayama automorphism $\nu$ of $\Ga$ up to isomorphism in Cases 1-4.
        \item  If both $A_{\Gamma}$ and $\gr(A_{\Gamma})$ are domestic in Cases 3-5, then $A_{\Gamma}$ is $n$-domestic if and only if $\gr(A_{\Gamma})$ is $n$-domestic for some positive number $n$.	
		\item In Case 5, $\Ga$ can not be determined by $\Ga/\ang{\nu}$ and the order $r$ of the Nakayama automorphism $\nu$ of $\Ga$ up to isomorphism, and it seems that the number $n$ in $n$-domestic is not directly related to the number $r$. For example, if $r=2$ and $\Ga/\ang{\nu}$ is the multiplicity-free ${\bf BG}$
$$\tikzset{every picture/.style={line width=0.75pt}}
\begin{tikzpicture}
\draw[-] (0.47,0.14) arc (15:165:0.5);
\draw[-] (-0.48,-0.14) arc (195:345:0.5);
\node at(-0.5,0) {$v$};
\node at(0.5,0) {$w$};
\node at(-0.5,0.5) {$b$};
\node at(0.5,0.5) {$b'$};
\node at(-0.4,-0.5) {$h_v$};
\node at(0.6,-0.5) {$h_w$};
\end{tikzpicture}$$
with two vertices $v$, $w$ and four half-edges $b$, $b'$, $h_v$, $h_w$, then $\Ga$ is isomorphic to $\Ga_{21}^{(e)}$ or $\Ga_{22}^{(e)}$, where $\Ga_{21}^{(e)}$ is the ${\bf bFBG}$
$$\tikzset{every picture/.style={line width=0.75pt}}
\begin{tikzpicture}[x=15pt,y=15pt,yscale=1,xscale=1]
\node at(0,0) {$a$};
\node at(5,0) {$b$};
\draw    (-0.3,0.1) .. controls (-4,3) and (1,3) .. (4.7,0.1) ;
\draw    (0.1,0.3) .. controls (1,3) and (6,3) .. (5.2,0.3) ;
\draw    (-0.1,0.3) .. controls (-2,4) and (9,4) .. (5.3,0) ;
\draw    (0.2,0.2) .. controls (2,2) and (5,2) .. (5,0.3) ;
\end{tikzpicture}$$
with $\bm(a)=\bm(b)=\frac{1}{2}$, and $\Ga_{22}^{(e)}$ is the ${\bf bFBG}$
$$\tikzset{every picture/.style={line width=0.75pt}}
\begin{tikzpicture}[x=15pt,y=15pt,yscale=1,xscale=1]
\node at(0,0) {$a$};
\node at(5,0) {$b$};
\draw    (-0.3,0.1) .. controls (-4.5,3) and (0.5,3) .. (4.7,0.1) ;
\draw    (-0.2,0.3) .. controls (-1.5,3) and (3.5,3) .. (4.8,0.3) ;
\draw    (0.2,0.3) .. controls (1.5,3) and (6.5,3) .. (5.2,0.3) ;
\draw    (0.3,0.1) .. controls (4.5,3) and (9.5,3) .. (5.3,0.1) ;
\end{tikzpicture}$$
with $\bm(a)=\bm(b)=\frac{1}{2}$. Note that $A_{\Ga_{21}^{(e)}}\cong\gr(A_{\Ga_{21}^{(e)}})$ is $2$-domestic but $A_{\Ga_{22}^{(e)}}\cong\gr(A_{\Ga_{22}^{(e)}})$ is $4$-domestic.
		 	
	\end{enumerate}
\end{Rem}

\subsection*{Acknowledgements}
The authors are supported by the National Natural Science Foundation of China \\ (No.12501050, No.12601072, No.12671044).

\end{document}